\documentclass[reqno,11pt]{amsart}
\usepackage{amsthm,amsfonts,amssymb,euscript,  
	mathrsfs,graphics,color,amsmath,latexsym,marginnote}  
\usepackage{cite}       
\usepackage{dsfont}          
\usepackage[english]{babel}
\usepackage{mathtools} 
\usepackage{times}

\numberwithin{equation}{section}

\usepackage{datetime}
\usepackage{lipsum}
\usepackage{enumerate}
\usepackage{tikz}
\usetikzlibrary{matrix}
\usepackage{hyperref}
\RequirePackage{amscd}
\RequirePackage{epic}
\RequirePackage[all]{xy}
\RequirePackage{url}
\RequirePackage[shortlabels]{enumitem}
\usepackage{empheq}
\usepackage{epsfig}
\usepackage{lineno} 
\usepackage{perpage}
\usepackage[english]{babel}

\usepackage[utf8]{inputenc}
\usepackage{cite}
\RequirePackage{amscd}
\RequirePackage{epic}
\RequirePackage{eepic}
\RequirePackage[all]{xy}
\RequirePackage{url}
\RequirePackage[shortlabels]{enumitem}
\usepackage{empheq}
\usepackage{nomencl}
\usepackage{epsfig}
\usepackage{graphicx}

\theoremstyle{plain}

\newtheorem{theorem}{Theorem}[section]
\newtheorem{proposition}[theorem]{Proposition}
\newtheorem{lemma}[theorem]{Lemma}

\newtheorem{remark}[theorem]{Remark}

\newcommand{\comment}[1]{}
   
\newcommand{\set}[1]{{\left\{#1\right\}}}
\newcommand{\norm}[1]{{\left |#1\right |}}
\newcommand{\normal}{\mathrm}

\RequirePackage{epic}
\RequirePackage{eepic}
\RequirePackage[all]{xy}
\RequirePackage{url}
\RequirePackage[shortlabels]{enumitem}
\usepackage{empheq}
\usepackage{nomencl}
\usepackage{epsfig}
\usepackage{graphicx}

\newcommand{\pa}[1]{{\left(#1\right)}}

\newcommand{\T}{\mathbb{T}}
\newcommand{\Z}{\mathbb{Z}}
\newcommand{\R}{\mathbb{R}}
\newcommand{\teta}{\theta}

\newcommand{\eps}{\varepsilon}

\renewcommand{\Im}{\operatorname{Im}}

\newcommand{\Mat}{\operatorname{Mat}}
\newcommand{\GL}{\operatorname{GL}}

\newcommand{\eqsys}[1]{{\left\{\begin{array}{ll}#1\end{array}\right.}}
\newcommand{\conj}{h\circ R_{2\pi\alpha}\circ h^{-1}}

\newcommand{\gr}[1]{\textbf{#1}}

\newcommand{\upp}[1]{\uppercase{#1}}

\newcommand{\id}{\operatorname{id}}

\usepackage{amsthm}

\newcommand{\C}{{\mathbb C}}

\newcommand{\cC}{\mathcal C}

\newcommand{\cG}{{\mathcal G}}
\newcommand{\cV}{{\mathcal V}}

\newcommand{\abs}[1]{\left| #1 \right|}

\newcommand{\sq}[1]{{\left[#1\right]}}
\newcommand{\Lip}{\operatorname{Lip}}
\newcommand{\Gr}{\operatorname{Gr}}

\makeatletter

\def\l@subsection{\@tocline{2}{0pt}{2.5pc}{5pc}{}}
\def\l@subsubsection{\@tocline{3}{0pt}{4.5pc}{5pc}{}}

\begin{document}

\title{Attractive invariant circles \`a la Chenciner}

\date{}

\author{Jessica Elisa Massetti}
\address{\scriptsize{Dipartimento di Matematica e Fisica, Universit\`a degli Studi RomaTre, Largo San Leonardo Murialdo 1, 00144}}
\email{jessicaelisa.massetti@uniroma3.it}


%

\begin{abstract}   
Studying general perturbations of a dissipative twist map depending on two parameters, a frequency $\nu$ and a dissipation $\eta$, the existence of a Cantor set $\mathcal C$ of curves in the $(\nu,\eta)$ plane such that the corresponding equation possesses a Diophantine quasi-periodic invariant circle can be deduced, up to small values of the dissipation, as direct consequence of a normal form Theorem in the spirit of  R\"ussmann and the ``elimination of parameters" technique. These circles are normally hyperbolic as soon as $\eta\not=0$, which implies that the equation still possesses a circle of this kind for values of the parameters belonging to a neighborhood $\mathcal V$ of this set of curves. Obviously, the dynamics on such invariant circles is no more controlled and may be generic, but the normal dynamics is controlled in the sense of their basins of attraction.\\
As it is expected, by classical graph-transform method we are able to determine a first rough region where the normal hyperbolicity prevails and a circle persists, for a strong enough dissipation $\eta\sim O(\sqrt{\eps}),$ $\eps$ being the size of the perturbation. Then, through normal-form techniques, we shall enlarge such regions and determine such a (conic) neighborhood $\cV$, up to values of dissipation of the same order as the perturbation, by using the fact that the proximity of the set $\mathcal C$ 
allows, thanks to R\"ussmann's translated curve theorem, to introduce local coordinates of the type (dissipation, translation) similar to the ones introduced by Chenciner in \cite{Chenciner:1985}. 
\end{abstract}  


\maketitle

\null\hfill
\textit{\`A Alain, mentor charmant et ami.}


\section{Introduction and results}

Let $\T = \R/2\pi\Z$. For $(\nu,\eta) \in\R\times I_{\eta_0},\, $ where $I_{\eta_0} = [-\eta_0,\eta_0]\setminus \set{0},\, \eta_0>0 $, let us consider the two parameters family of twist maps of the cylinder 
\begin{equation}\label{diffeo P}
P: \T\times\R \to  \T\times\R \quad \quad P(\teta,\rho)= (\teta + 2\pi\nu + \frac{1-e^{-2\pi\eta}}{\eta}\rho, \rho e^{-2\pi\eta}). 
\end{equation}
The circle $C_0 = \T\times \set{\rho=0}$ is invariant for $P$, normally attractive (resp. repulsive) whenever $\eta >0$ (resp. $\eta < 0)$, and carries the linear dynamics $\teta \mapsto \teta + 2\pi\nu$. Note that the above maps, which in the regime $\eta\in (0,\eta_0]$ are in the class of the so-called \emph{dissipative} twist maps, extend continuously to symplectic twist ones as $\eta\to 0$ and arise naturally as time $2\pi$-maps of the classic dissipative spin-orbit model \footnote{This problem concerns the rotations of a tri-axial non rigid satellite (or planet), whose center of mass revolves on a given Keplerian elliptic orbit, focused on a fixed massive point.} in Celestial mechanics, before taking into account possible non-autonomous perturbations periodic in time, see \cite{Murray-Dermott}. 
It is worth to mentioning that under appropriate hypotheses on the internal structure, rotation and shape of the body, the non-conservative part of the unperturbed motion is linear in the friction. The perturbing interactions that are classically considered, already very rich in dynamical implications, are of Hamiltonian nature, which makes the corresponding system of equations conformally symplectic \cite{Calleja-Celletti-delaLlave:2013}.\footnote{The literature on the astronomical model is quite huge, we send the reader interested in the over well studied model with Hamiltonian perturbations to the recent \cite{CallejaArxiv, Correia-Delisle} and references therein.} \\However, in the present paper, we shall make use of the equation \eqref{diffeo P} as an expedient for having a map that enjoys nice properties at rest, then we shall study some aspects of its dynamical behavior in the frame of very general, non-Hamiltonian, perturbations so that no special structure is present. Therefore, in  this ``\textit{exercice de pens\'ee}"  we shall propose a method  that can be adapted to very general situations, in the same spirit as Chenciner's works on the bifurcations of elliptic fixed points \cite{Chenciner:1985, Chenciner:1985a, Chenciner:1988}. Let us describe the main features that will guide our study.

\smallskip

By introducing a Diophantine $\alpha\in\R$ and performing the variable's change $(\teta,\rho)\mapsto (\teta, \rho + \nu - \alpha = r)$ the map $P$ reads
\begin{equation}\label{Diophantine P}
	\eqsys{P(\teta,r) = (\teta',r')\\
		\teta' = \teta + 2\pi\nu + \frac{(1 - e^{-2\pi\eta})}{\eta}(r - (\nu-\alpha))\\
		r' = (\nu -\alpha) + e^{-2\pi\eta}(r - (\nu -\alpha))\,,}
\end{equation}
and we see that each circle of radius $r$ is ``translated" by the quantity 
\begin{equation*}
	\tau'= r' - r= (e^{-2\pi\eta}-1)(r - (\nu - \alpha))
\end{equation*}
and ``rotated" by the angle 
\begin{align*}
	\teta' -\teta &= 2\pi\nu + [r - (\nu - \alpha)]\frac{1 - e^{-2\pi\eta}}{\eta} \\
	&= 2\pi\nu - \frac{\tau'}{\eta}.
\end{align*}
In particular, the unique circle of rotation $2\pi\alpha$ is the one with radius
\begin{equation}\label{circle dio}
	r_\alpha = (\nu - \alpha)\sq{1 + \frac{2\pi\eta}{e^{-2\pi\eta} -1}},
\end{equation}
which is translated by a quantity
\begin{equation*}
	\tau_\alpha= 2\pi\eta (\nu - \alpha).
	\label{translation}
\end{equation*}

\smallskip
The (unique) choice of the frequency-parameter $\nu = \alpha$ yields an invariant circle $(\teta, r_\alpha = 0)$ of Diophantine rotation $2\pi\alpha$, whose normal dynamics is  characterized by $\eta$. \\
This calls for an important comment.
In general, if we are given a twist map $P'$ that preserves an invariant $2\pi$-Diophantine circle and ask if some $Q'$ close enough to $P'$ still admits an invariant circle on which the dynamics is conjugated to $\teta\mapsto  \teta + 2\pi\alpha$, the celebrated \textit{theorem of the translated curve}  of R\"ussmann \cite{Russmann:1970} tells us that in non-conservative regime, the persistence of such a Diophantine invariant circle  is a phenomenon of co-dimension 1: in general the
invariant curve  does not persist but it is translated in the normal direction\footnote{for a precise statement see Theorem \ref{Russmann theorem}}. Of course, if  one restricts the study to twist maps that enjoy the intersection property or that are exact symplectic (as the ones in Moser's seminal work \cite{Moser:1962}), then the persistence of each Diophantine circle follows automatically from R\"ussmann's result\footnote{According to an anecdote told by Herman during a lecture on the topic, it seems that R\"ussmann found its ``translated curve" while trying to give an alternative proof of Moser's invariant curve theorem}, the area conservation filling
the co-dimension gap in the dynamical conjugacy with a map that preserves the Diophantine circle, see \cite{Bost:1985}.

\smallskip
 
Given $\eps>0$, and given  $ f, g$ two real analytic maps from $\R^2$ to $\R$ that are $2\pi$-periodic in the first variable, with uniformly bounded $C^2$-norm,  we consider the following $\eps$-family of real analytic diffeomorphisms of the annulus

\begin{equation}\label{perturbation}
{Q}(\teta,\rho) = (\teta + 2\pi\nu + \frac{1-e^{-2\pi\eta}}{\eta}\rho + \eps f(\teta,\rho), \rho e^{-2\pi\eta} + \eps g(\teta,\rho))\,.
\end{equation}

For values of $\eps\ll 1$, we shall determine in the space  $(\eps,\eta,\nu)$ , regions of parameters in correspondence of which the maps possess invariant normally attractive (or repulsive) circles, either carrying a Diophantine rotation or a dynamics that is not necessarily controllable.
Our methods are reminiscent of the ideas introduced by A. Chenciner in a series of articles of the $80$'s and, concerning the specific content of the present paper, in the study he carried out in \cite{Chenciner:1985}.

\medskip 

More specifically, in $1985$ A. Chenciner started an analysis of the dynamical properties of generic $2$-parameter  families  of germs of diffeomorphism of $\R^2$ which unfold an elliptic fixed point. In \cite{Chenciner:1985, Chenciner:1985a, Chenciner:1988}, he showed that along a certain curve $\Gamma$ in the space of parameters, we find all the complexity that the dynamics of a germ of generic area preserving diffeomorphism of $\R^2$ presents, in the neighborhood of an elliptic fixed point. Inspired by the first of these works treating the case of invariant curves, the present paper proposes a similar study for the general perturbations $Q$ introduced above. 

In what follows we shall place ourselves in low dissipative regime $0 <\eta \ll 1$ and reason on attractive circles. Analogue results hold for the case of repulsive ones, i.e. the case $\eta < 0$. We shall denote with $\gtrdot$ and $\lessdot$ inequalities that may depend on pure constants.
\noindent
 Our study can be described as follows.\\
 
\noindent
$\bullet$ We first prove that for high enough values of the hyperbolicity $\eta$ (i.e. $\eta\sim \sqrt{\eps}$), the diffeomorphism $Q$ still admits an invariant circle, close to $\rho=0$,
 which is indeed a graph. We show this by means of the graph-transform method, which loosely consists in proving the existence of the invariant graph-curve as the fixed point of a contraction map on a well chosen functional space (of Lipschitz graphs, that a posteriori are in fact $C^1$). The parameters at play\footnote{ i.e. dissipation, perturbation, Lipschitz constant of the graph we look for.} must satisfy certain compatibility conditions in terms of their size so that the above (graph transform) map is well defined on the chosen functional space (see Proposition \ref{condition GT1}). As a consequence, we can detail a first region of parameters in correspondence of which the map admits an invariant circle. This is the content of Section \ref{section graph transf}. The first main result can simply be stated as follows. 

\begin{theorem}
 If $\eta\gtrdot\sqrt{\eps}$, the map $Q$ possesses a unique invariant attractive $C^1$-circle in the vicinity of $C_0 = \T\times \set{\rho=0}$, whose basin of attraction coincides with $\T\times\R$.
	\label{theorem A}
\end{theorem}

\noindent
$\bullet$ Secondly, in line with the previous comments on $P$, in the same spirit as A. Chenciner\footnote{In the cited paper the rotation and parameter of normal hyperbolicity were denoted by $\omega$ and $\chi$ respectively.} in \cite{Chenciner:1985} we introduce a Diophantine $\alpha$ and write the diffeomorphism in terms of the translation $\tau_\alpha$. After an appropriate change of coordinates that localizes us close to the (translated) circle of rotation $\alpha$ (recall \eqref{circle dio}) the perturbation $Q$ is in a form like
\begin{equation}
	Q(\teta,r)= \pa{\teta + 2\pi\alpha + \frac{1- e^{-2\pi\eta}}{\eta}\,{r} + \eps f(\teta,r),r e^{-2\pi\eta} + \tau_\alpha + \eps g(\teta,r)}.
\end{equation}
Since $Q$ is twist and $\alpha$ is Diophantine, R\"ussmann's Theorem entails that for small enough $\eps$, there exists a circle $(\teta,\gamma(\teta))$ whose image by $Q$ is translated by a quantity $\lambda\in\R$, necessarily of the form $\lambda = \tau_\alpha + O(\eps)$, and on this circle the dynamics is conjugated to a Diophantine rotation $\teta\mapsto \teta + 2\pi\alpha$ by a diffeomorphism of $\T$ close to the identity.
 From this, we then remark that for any admissible perturbation and for values of $\eta \gtrdot \eps$,  there exists a curve  $\cC_\alpha$ $(\eps,\eta)\mapsto \nu(\eps,\eta)$ in the $(\eta,\nu)$-plane of parameters in correspondence of which the translation $\lambda$ vanishes, hence $Q$
admits an attractive Diophantine invariant circle of rotation $2\pi\alpha$ for those values of the parameters. This zero step allows to recover in this general context a geometry of the space of parameters similar also to the one that was given in \cite[Sec. 6]{Massetti:ETDS} in the special frame of Hamiltonian perturbations.

\begin{figure}[h]
\centering
\vskip-2mm
\includegraphics[width=.5\textwidth]{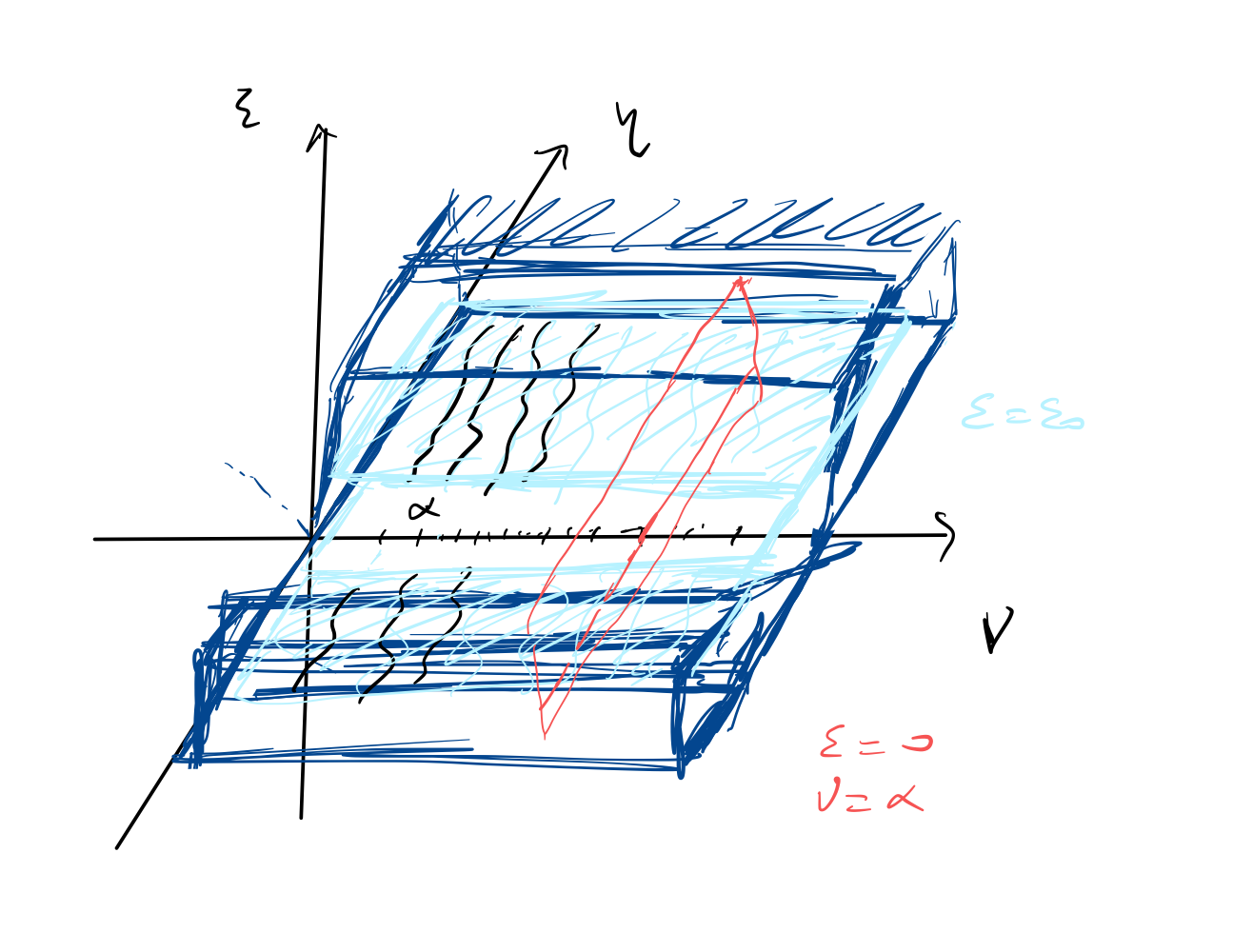}
\vskip-10mm
\caption{In light blue it is indicated the plane $(\eta,\nu)$ corresponding to a fixed admissible value $\eps_0$ of $\eps$,  while in black the Cantor of curves $C_\alpha$.}
\end{figure}

 In the present situation, the existence of such ``good curves" $\cC_\alpha$ follows from a straightforward application of the normal form result \cite[Theorem 5.4]{Massetti:APDE}, which is recalled in the Appendix, and the classical implicit function theorem, see Remark \ref{rmk parameters}  and Appendix \ref{appendix curves}. This is done in the very first part of Section \ref{second localization}. \\
$\bullet$ Because of the normal dynamics of these invariant circles, the diffeomorphism must possess an invariant circle also for values of the parameters belonging to a small neighborhood $\cV$ of each curve $\cC_\alpha$. With the aim of determining quantitatively such $\cV$, we then study what happens in the complementary of this set of curves of parameters, where the circle is a priori only translated. For this, we localize the diffeomorphism \emph{at} the translated circle of rotation $2\pi\alpha$, see Subsection \ref{second localization a 0}, and following the same strategy as Chenciner in \cite{Chenciner:1985} we use again the Diophantine properties of $\alpha$ in order to perform several variable's change that eventually put the diffeomorphism in a meaningful normal form (see Proposition \ref{NF}) up to\footnote{We denoted $\Theta, R$ the final variables of $Q$ after the changes} a certain order $O(|R|^k)$, where the perturbative terms are reduced in size and the only terms that may purely depend on the angles  vanish with the translation $\lambda$. This enables us to
determine quantitatively such a (cone) neighborhood $\mathcal V$ of the curves $\cC_\alpha$'s,  up to values of dissipation of the same order as the perturbation, thus enlarging the region that has been determined in Theorem \ref{theorem A}. In such a region the normal hyperbolicity is strong enough with respect to the remainders so that  the existence of an invariant normally hyperbolic circle by graph transform again is implied again. In this procedure, it is fundamental the fact that the proximity of the set $\mathcal C$, thanks to Russman's translated curve theorem, enables us to introduce local coordinates of the type (dissipation, translation) similar to the ones introduced by Chenciner. 
The following theorem holds.
\begin{theorem}\label{theorem NH2}
	\comment {If $\alpha$ is Diophantine, there exists a real analytic change of variables that transforms $Q$ into the normal form:
		\begin{equation}
			\eqsys{   Q(\Theta,R)= (\Theta',R')\\
				\Theta'= \Theta + 2\pi\alpha + \sum_{i=1}^k \bar\alpha_i R^i + O(\eps \abs{R}^{k+1}) + O(\abs{\lambda}\eps)\\
				R' = \lambda + \bar\beta_1\,R + \sum_{i=2}^k \bar\beta_i\, R^i + O(\eps\abs{R}^{k+1})  + O(\abs{\lambda}\eps),
			}   
			\label{NF}
		\end{equation}
		$\bar\alpha_i$ and $\bar\beta_i$ being constants, $O(\eps\abs{R}^{k+1})$ being terms of order $\geq k +1$ vanishing when $R = 0$, $O(\abs{\lambda}\eps)$ being terms of order $O(\tau\eps) + O(\eps^2)$, possibly depending on $\Theta$ but vanishing when $\lambda = 0$. \\}
	There exists an $\eps_0>0$ such that every Diophantine $\alpha$ identifies a region in the space of parameters $(\eps,\eta,\nu) = [0, \eps_0]\times \R \times \R$ defined by
	\begin{equation*}
		\eqsys{ \eta\gtrdot\eps \\
			\eta \geq \sqrt{2\pi}\abs{\nu - \alpha}, }
	\end{equation*}   
	where $Q$ possesses an invariant normally-attractive  $C^1$-circle, whose basin of attraction coincides with $\T\times\R$.
	
\end{theorem}

\begin{figure}[h]
\centering
\vskip-2mm
\includegraphics[width=.4\textwidth]{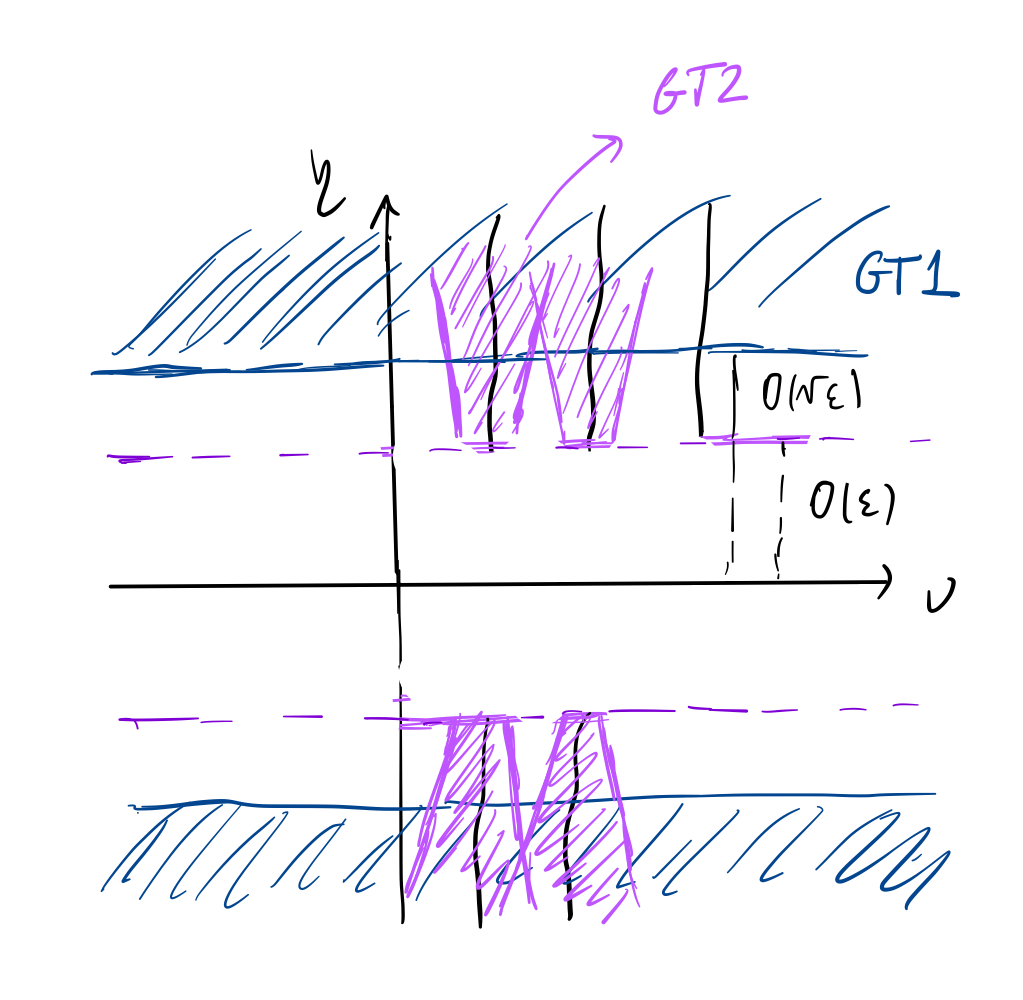}
\vskip-10mm
\caption{\footnotesize{Using the diophantine properties of $\alpha$, we reduce the size of remainders and determine regions where the graph transform works again, but for smaller values of $\eta$ in terms of $\eps$.} }\label{gt12}
\end{figure}

Some other drawings\footnote{whose beauty is far from being comparable to the one of Chenciner's drawings!} in the text portray the geometry of the good zones of hyperbolicity.
 We believe that the relevance of the approach proposed here stands in the fact that is very general,  and allows to handle generic non-conservative systems,  when no special symmetries helps in the study. Of course, the price to pay is a loss in sharpness that very special situations may favorite. Moreover, an extension to high dimension would be quite straightforward, thus giving space to the study of more general situations.

\section{Invariant circles through graph transform}
\label{section graph transf}

It is well known that invariant manifolds of dimension greater than $0$ systematically persist only if they are "normally hyperbolic" (see \cite{Hirsch-Pugh-Shub:1977} and \cite{Mane:1978}). Roughly speaking, if $P$ is a $C^1$-diffeomorphism of a smooth manifold $M$ and $V\subseteq M$ a submanifold invariant by $P$, $V$ is \emph{normally hyperbolic} if the tangent map $TP$ restricted to the normal direction to $V$ dominates the restriction of $TP$ to the tangent direction $TV$ \footnote{This property is extensively used in bifurcation theory (Hopf bifurcation, Denjoy attractors for example; see \cite{Chenciner-Iooss:1979} and references therein for instance).} \\ Hence, if the normal attraction is strong enough, Hadamard's "graph transform" method  allows to find the invariant perturbed manifold as the fixed point of a contraction in a well chosen functional space. See \cite{Shub:1987}. \smallskip\\ 


\label{Normally hyperbolic regime}
There are many equivalent definitions of normal hyperbolicity (see \cite{Hirsch-Pugh-Shub:1977, Shub:1987, Berger-Bounemoura}), for convenience of the reader we recall the strongest one, given by Hirsh-Pugh-Shub in \cite{Hirsch-Pugh-Shub:1977}.  Let $V$ be a smooth compact submanifold of a smooth manifold M and suppose that $P: M\to M$ is a $C^r$ diffeomorphism and $P(V) = V$. \\ We say that $P$ is \emph{$r$-normally hyperbolic} at $V$ (or that $V$ is $r$-normally hyperbolic) if the tangent bundle of $M$ over $V$ has a $TP$-invariant splitting 
\begin{equation}
	TM_{|V} = TV \oplus N^s \oplus N^u,
	\label{splitting}
\end{equation}
and for all $x\in V$, 
\begin{equation} 
	\eqsys{\sup_x \norm{TP_{|N_x^s}} < \inf_x m(TP_{|T_x V})^k,\\
		\\
		\inf_x m(TP_{|N_x^u}) > \sup_x \norm{TP_{|T_x V}}^k\,\qquad 1\leq k \leq r. }
	\label{domination}
\end{equation}
where $m(\cdot )$ is the minimum norm.\footnote{The minimum norm of a linear transformation $A$ is defined as $m(A)= \inf_{\abs{y}=1} \abs{Ay}$.}\\  Condition \eqref{domination} expresses domination of the normal behavior over the tangent behavior of $TP$.\\
In the case of our interest this definition becomes trivial. In fact, for the real analytic diffeomorphism $P$ of $\T\times\R$ defined by \eqref{diffeo P},  the circle $C_0 = \T\times \set{\rho=0}$ is clearly $r$-normally hyperbolic; the tangent bundle being trivial the given definition hides no ambiguity and the invariant splitting \eqref{splitting} is reduced to the sum of the tangent bundle $TC_0$ and a normal attractive one (resp. repulsive when $\eta<0$). \\
For simplicity, in the following we shall consider the case $r=1$. 
 Let now consider the perturbation
\begin{equation} \label{Q perturbato}
 Q(\teta,\rho)= (\teta + 2\pi\nu + \frac{1-e^{-2\pi\eta}}{\eta}\rho + \eps f(\teta,\rho), \rho e^{-2\pi\eta} + \eps g(\teta,\rho)),
 \end{equation}
   where $f$ and $g$ are bounded real analytic functions in their arguments, of uniformly bounded $C^2$ norms. Note that we are interested in small values of $\eta \ll 1$, hence the twist term $ \frac{1-e^{-2\pi\eta}}{\eta}$ and the contraction coefficient $e^{-2\pi\eta}$ are $O(1)$.
\begin{remark}There exists a positive constant $C$ such that $\rho' - \rho < 0$ outside the annulus $\abs{\rho}\leq C \eps/\eta$; the normal behaviors of $Q$ and $P$ are hence analogue. It remains to refine the study inside of it and conclude the existence of an invariant circle.
\end{remark}

\textbf{Proof of Theorem \ref{theorem A}}
The proof is classical, but we shall detail it and articulate it in several steps. \\ Since $\rho=0$ is the only invariant circle of $P$, we will find a unique invariant circle of $Q$ as the fixed point of a graph transform on an appropriate functional space; the dissipation makes the graph transform a contraction.\footnote{In this context we could take the term "normal hyperbolic" as the synonym of the contraction property of the graph transform associated to $Q$.}

\textit{Step 1: Definition of the graph transform.} Let us take the compact $\T\times [-1,1]$ centered at $\rho=0$.  We shall identify $2\pi$-periodic Lipschitz maps $\varphi : \R\to [-1,1]$ with $\Lip \varphi\leq k$, to Lipschitz maps $\varphi : \T\to [-1,1],\, \teta\mapsto \varphi(\teta)$. We will call $\Lip_{k}$ the set of Lipschitz functions with Lipschitz constant less than or equal to $k$, and endow it with the $C^0$-metric. \\Let $\Gr\varphi =\set{(\teta,\varphi(\teta))\in \T\times [-1,1]}$ be the graph of $\varphi$. Since $Q$ is defined everywhere, the composition $Q(\Gr\varphi)$ makes sense. \\ We note $Q(\teta,r)=(\Theta, R)$. The components of $Q(\teta,\varphi(\teta))$ are: 
\begin{align*}
	\Theta\circ(\id,\varphi)(\teta) &= \teta + 2\pi\nu + \frac{1-e^{-2\pi\eta}}{\eta}\varphi(\teta) + \eps f(\teta,\varphi(\teta))\\
	R\circ(\id,\varphi)(\teta) &= \varphi(\teta)e^{-2\pi\eta} + \eps g(\teta,\varphi(\teta)).
\end{align*}
We define as usual the \emph{graph transform} $\Gamma: \varphi\mapsto \Gamma\varphi$ by
\begin{equation}
	\Gamma\varphi : \teta\mapsto R\circ (\id,\varphi)\circ [\Theta\circ(\id,\varphi)]^{-1} (\teta).
\end{equation}
The graph of $\Gamma\varphi$ is the image by $Q$ of the graph of $\varphi$: $Q(\Gr\varphi)=\Gr(\Gamma\varphi)$ (see figure below); it is a classical tool for proving the existence of invariant attracting objects.

\smallskip

We look for a class of Lipschitz functions $\Lip_k$ such that $\Gamma$ defines a contraction of $\Lip_k$ in the $C^0$-metric. \\Although we are interested in small values of $k>0$ ($\eps$ being small, we do not expect the invariant curve to be in a class of functions with big variations) we will need $k$ as well as $\eta$ to be larger than $\eps$.\\ We shall realize this for $1\gg\eta,k,\eps$, since if $\eta$ is in the vicinity of $1$, the persistence of the invariant circle would be very easily shown. 
\begin{remark} The graph transform method provides an invariant circle which a priori is no more than Lipschitz (in the sense that it will be the graph of a Lipschitz map of Lipschitz constant $\leq k$) but, as it is well known this invariant circle is proved to be automatically at least $C^1$. We send the reader to \cite[Sec. 4]{Hirsch-Pugh-Shub:1977} for a  proof of this fact. Moreover, when $P$ is $r$-normally hyperbolic and $Q$ is in a $C^r$-neighborhood of $P$, the invariant circle of $Q$ is proved to be $C^r$ (see again \cite[ 2-3]{Hirsch-Pugh-Shub:1977} or \cite[Corollary 2.2]{Berger-Bounemoura}).\end{remark}
\begin{figure}[h]
\begin{picture}(0,0)%
\includegraphics{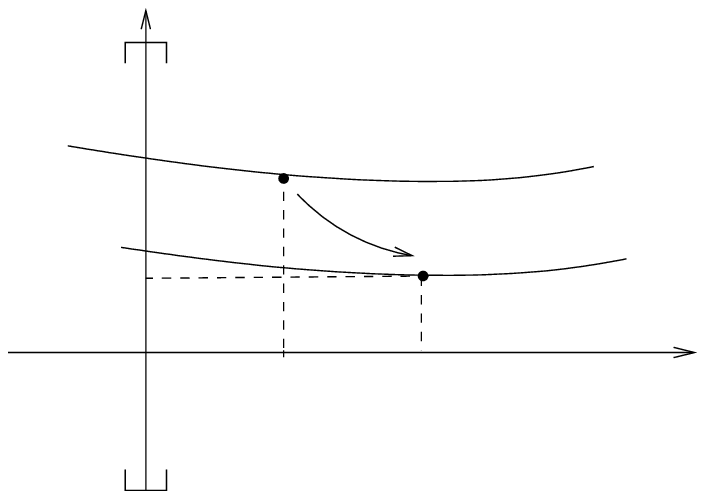}%
\end{picture}%
\setlength{\unitlength}{2901sp}%
\begingroup\makeatletter\ifx\SetFigFont\undefined%
\gdef\SetFigFont#1#2#3#4#5{%
	  \reset@font\fontsize{#1}{#2pt}%
	  \fontfamily{#3}\fontseries{#4}\fontshape{#5}%
	  \selectfont}%
\fi\endgroup%
\begin{picture}(4755,3423)(2551,-5080)
\put(5281,-4486){\makebox(0,0)[lb]{\smash{{\SetFigFont{8}{9.6}{\familydefault}{\mddefault}{\updefault}{\color[rgb]{0,0,0}$\Theta(\theta,\varphi(\theta))$}%
				}}}}
\put(4321,-4441){\makebox(0,0)[lb]{\smash{{\SetFigFont{8}{9.6}{\familydefault}{\mddefault}{\updefault}{\color[rgb]{0,0,0}$\theta$}%
				}}}}
\put(6661,-2596){\makebox(0,0)[lb]{\smash{{\SetFigFont{8}{9.6}{\familydefault}{\mddefault}{\updefault}{\color[rgb]{0,0,0}$(\theta,\varphi(\theta))$}%
				}}}}
\put(6946,-3586){\makebox(0,0)[lb]{\smash{{\SetFigFont{8}{9.6}{\familydefault}{\mddefault}{\updefault}{\color[rgb]{0,0,0}$(\theta,\Gamma\varphi(\theta))$}%
				}}}}
\put(2566,-3661){\makebox(0,0)[lb]{\smash{{\SetFigFont{8}{9.6}{\familydefault}{\mddefault}{\updefault}{\color[rgb]{0,0,0}$R(\theta,\varphi(\theta))$}%
				}}}}
\put(5086,-3256){\makebox(0,0)[lb]{\smash{{\SetFigFont{8}{9.6}{\familydefault}{\mddefault}{\updefault}{\color[rgb]{0,0,0}$Q$}%
				}}}}
\put(7291,-4111){\makebox(0,0)[lb]{\smash{{\SetFigFont{8}{9.6}{\familydefault}{\mddefault}{\updefault}{\color[rgb]{0,0,0}$\mathbb{T}$}%
				}}}}
\put(3376,-1816){\makebox(0,0)[lb]{\smash{{\SetFigFont{8}{9.6}{\familydefault}{\mddefault}{\updefault}{\color[rgb]{0,0,0}$\mathbb{R}$}%
				}}}}
\put(3106,-2176){\makebox(0,0)[lb]{\smash{{\SetFigFont{8}{9.6}{\familydefault}{\mddefault}{\updefault}{\color[rgb]{0,0,0}$1$}%
				}}}}
\put(3061,-5011){\makebox(0,0)[lb]{\smash{{\SetFigFont{8}{9.6}{\familydefault}{\mddefault}{\updefault}{\color[rgb]{0,0,0}$-1$}%
				}}}}
\end{picture}%
\caption{How the graph transform acts}
\end{figure}

\textit{Step 2: Well-definition of $\Gamma_\varphi$ as a map on $\Lip_k$}. The following Lemma is classical and the proof left in the Appendix. The well-definition of the graph transform will be subordinated to an appropriate relation among the paramenters.
\begin{lemma}	\label{inversion lemma} The following holds.
\begin{enumerate}
	\item For every positive $\eta$, provided $\eps$ is sufficiently small, $\Theta\circ (\id,\varphi)$ is invertible.
	\item The functions $\Theta$ and $R$ are Lipschitz on $\T\times [-1,1]$.
	\end{enumerate}
\end{lemma}

\comment{ \begin{lemma}
		The graph transform $\Gamma$ is well defined from $\Lip_k$ to itself, provided $k$ and $\eta$ satisfy $\eps/\eta< k \ll\eta\ll1$.
	\end{lemma}
	
	\begin{proof}
		From the definition of the graph transform and the previous lemmata, we have
		\begin{align*}
			\abs{\Gamma\varphi(\teta_1) - \Gamma\varphi(\teta_2)}&\leq \frac{\Lip R\circ(\id,\varphi)}{1-\Lip u}\abs{\teta_1 - \teta_2}\\
			&\leq \frac{ke^{-2\pi\eta} + \eps A_g(1+k)}{1 - \sq{\frac{1-e^{-2\pi\eta}}{\eta}\,k +\eps A_f(1 + k)}}\abs{\teta_1 - \teta_2}.
		\end{align*}
		We want to find conditions on $\eta$ and $k$, such that $\eps\ll1$ being fixed, $\Gamma$ is well defined in $\Lip_k$; we must satisfy 
		\begin{equation*}
			ke^{-2\pi\eta} + \eps A_g(1+k) \leq k\set{1 - \sq{\frac{1-e^{-2\pi\eta}}{\eta}\,k +\eps A_f(1 + k)}},
		\end{equation*} 
		hence 
		\begin{equation*}
			k\set{1 - e^{-2\pi\eta} - \sq{\frac{1-e^{-2\pi\eta}}{\eta}\,k +\eps A_f(1 + k)}}\geq \eps A_g (1 + k).
		\end{equation*}
		It suffices to choose $k$ so that\footnote{Note that, despite one could expect $\eta$ to be of order $\eps$, in this case the $O(1)$-size of the twist $\frac{1-e^{-2\pi\eta}}{\eta}$ forces $k$, hence $\eta$, to be strong enough to satisfy the inequality.} 
		\begin{equation}
			1\gg\eta \gg k \quad \text{with}\quad k>\frac{\eps}{\eta}.
			\label{condizione k}
		\end{equation}
		Clearly the larger $\eta$ is, the easier it is to realize the inequality.
	\end{proof}
	
}

The following Proposition give the compatibility conditions on the parameters, for the well definition of the graph transform.
\begin{proposition}	\label{condition GT1}
	Let $0 < \eta,k, \eps \ll 1$. The graph transform $\Gamma$ is well defined from $\Lip_k$ to itself, if the following relations between the perturbation $\eps, k$ and $\eta$ are satisfied 
	\begin{equation}
		\eps\leq\frac{\pi}{6A}\eta k,\quad k\leq \frac{1}{6} \eta,\quad  \eta\leq c\ll 1,
	\end{equation}
	where $A= A_f + A_g$, $A_f$ and $A_g$ being the sup norm on the derivatives of $f,g$, and $c$ is a positive constant.
\end{proposition}

\begin{proof}
	Let $\Theta\circ(\id,\varphi) =: \id + u$. From the definition of the graph transform and Lemma \ref{inversion lemma}, we have
	\begin{align*}
		\abs{\Gamma\varphi(\teta_1) - \Gamma\varphi(\teta_2)}&\leq \frac{\Lip\pa{ R\circ(\id,\varphi)}}{1-\Lip u}\abs{\teta_1 - \teta_2}\\
		&\leq \frac{ke^{-2\pi\eta} + \eps A_g(1+k)}{1 - \sq{\frac{1-e^{-2\pi\eta}}{\eta}\,k +\eps A_f(1 + k)}}\abs{\teta_1 - \teta_2}.
	\end{align*}
	In order to prove that $\Gamma\varphi\in\Lip_k$, we need to find conditions on $\eta,$ $k$ and $\eps$  such that is the following inequality is satisfied
	\begin{equation*}
		ke^{-2\pi\eta} + \eps A_g(1+k) \leq k\set{1 - \sq{\frac{1-e^{-2\pi\eta}}{\eta}\,k +\eps A_f(1 + k)}},
	\end{equation*} 
	hence 
	\begin{equation}
		k\set{1 - e^{-2\pi\eta} - \sq{\frac{1-e^{-2\pi\eta}}{\eta}\,k +\eps A_f(1 + k)}}\geq \eps A_g (1 + k).
		\label{disuguaglianza eta}
	\end{equation}
	Developing in $\eta$ (remember that $\eta \ll 1$) and dividing by $k$, we get
	\[2\pi\eta + O(\eta^2) + 2\pi^2\eta k + O(\eta^2)k \geq 2\pi\ k + \frac{\eps}{k} A_g + \eps A + \eps k A_f,\] where the constant $A = A_f + A_g$.
	It is immediate to see that  there exist a positive $c'$ such that if $\eta$ and $k$ satisfy 
	\begin{equation}
		1\gg\eta > c'k \quad \text{together with}\quad k>\frac{\eps}{\eta} \frac{A}{2\pi}.
	\end{equation}
	the inequality would be easily satisfied. \\ Finally, since there exists a positive constant $C'$ such that $$2\pi\eta + O(\eta^2)\geq 2\pi\eta - C'\eta^2\geq \pi\eta$$ when $\eta$ is sufficiently small, the condition $1\gg \eta \geq 6\max\pa{k, \frac{\eps}{k}\frac{A}{\pi}}$ suffices.
	
	\comment{
		!!!! Non so se metterlo per esibire una scelta o meno !!!!!
		In particular, there exist two positive constants $\beta$ and $\gamma$ such that \[2\pi\eta - 2\pi^2\eta^2 + O(\eta^3) - 2\pi k  + 2\pi^2\eta k + O(\eta^2)k \geq 2\pi\eta - 2\pi^2\eta^2 - \gamma \eta^3 + (2\pi^2\eta - 2\pi - \beta\eta^2) k\]
		and, since $$2\pi\eta - 2\pi^2\eta^2 - \gamma\eta^3 \geq \frac{\eta}{2\pi}\quad \text{ and }\quad 2\pi^2\eta - \beta\eta^2\geq \pi\eta$$ if $\eta\leq \min\pa{\frac{-\pi^2 + \sqrt{\pi^4 + \gamma\frac{4\pi^2-1}{2\pi}}}{\gamma}, \frac{2\pi^2 - \pi}{\beta}}$, the right hand side satisfies
		\[ 2\pi\eta - 2\pi^2\eta^2 - \gamma \eta^3 + (2\pi^2\eta - 2\pi - \beta\eta^2) k\geq \frac{\eta}{2\pi} - \delta k,\] where $\delta$ is a positive constant $\pi\leq\delta\leq 2\pi$, if $\eta\leq 1$. The choice \[\eta = 6\pi \max\pa{\frac{\eps}{k} A, \eps A, 2\pi k}\] thus suffices.
	}
	Clearly the larger $\eta$ is, the easier it is to realize \eqref{disuguaglianza eta}.
\end{proof}

\begin{remark}[Necessarity of $\eta\geq C\sqrt{\eps}$]
	Condition \eqref{condition GT1} entails that $\eta$ is of order $O(\sqrt{\eps})$ at least. In this frame of general perturbations, where no cancellations occur a priori, we cannot improve this order. It will be necessary to use more refined normal form techniques, as we shall show in the next Section.\\
	In fact, if inequality \eqref{disuguaglianza eta} is satisfied, when $\eta\ll 1$ ($\eta\leq 1/2\pi$ would suffice), by expanding $e^{-2\pi\eta}$ and $\frac{1-e^{-2\pi\eta}}{\eta}$ the following chain holds
	\begin{align*}
		2\pi\eta - 2\pi^2\eta^2 + O(\eta^3) - 2\pi k + 2\pi^2 \eta k + O(\eta^2) k &\geq \frac{\eps}{k} A_g  + \eps A + \eps k A_f \\
		2\pi\eta - 2\pi^2\eta^2 + O(\eta^3) + O(\eta^2)k &\geq (2\pi - 2\pi^2\eta) k + \frac{\eps}{k} A_g + \eps A + \eps k A_f \geq \\ \geq \pi k + \frac{\eps}{k} A_g + \eps A + \eps k A_f.
	\end{align*}
	Moreover, since there exist two positive constants $\beta, \gamma$ such that $O(\eta^2)k \leq \beta\eta^2 k$ and $O(\eta^3)\leq \gamma \eta^3$, if $\eta\leq \sqrt{\frac{\pi}{2\beta}}$ we can write 
	\begin{equation*}
		2\pi\eta - 2\pi^2\eta^2 + \gamma\eta^3  \geq (\pi -\beta\eta^2)k + \frac{\eps}{k} A_g + \eps A + \eps k A_f \geq \frac{\pi}{2} k  + \frac{\eps}{k} A_g + \eps A + \eps k A_f.
	\end{equation*}
	Eventually note that if $\eta\leq \min \pa{\sqrt{\frac{\pi}{2\beta}}, \frac{2\pi^2}{\gamma}},$  $- 2\pi^2\eta^2 + \gamma\eta^3 $ is negative  and
	\[
	2\pi\eta \geq \frac{\pi}{2} k + \frac{\eps}{k} A_g + \eps A + \eps k A_f.\]
	Thus, in the regime $\eta\ll 1$ we are interested at, where in particular $\eta \leq \min \pa{\sqrt{\frac{\pi}{2\beta}}, \frac{2\pi^2}{\gamma}, \frac{1}{2\pi}}$, $\eta$ is necessarily of order at least $\sqrt{\eps}$.  
\end{remark}

The following technical lemma is the key for showing that the graph transform defines a contraction in the space $\Lip_k$.
\begin{lemma}
	\label{lemma per contrazione}
	Let $z=(\teta,\rho)$ be a point in $\T\times [-1,1]$ and let $\eta,k,\eps$ satisfy condition \eqref{condition GT1}. The following inequality holds for every $\varphi\in\Lip_k$: \[\abs{R(\teta,\rho) - \Gamma\varphi\circ\Theta(\teta,\rho)}\leq C \abs{\rho - \varphi(\teta)},\] $C$ being a constant smaller than $1$.
\end{lemma}
\begin{proof}
	The following chain of inequalities holds:
	\begin{align*}
		\abs{R(\teta,\rho) - \Gamma\varphi\circ\Theta(\teta,\rho)} &\leq \abs{R(\teta,\rho) - R(\teta,\varphi(\teta))} + \abs{R(\teta,\varphi(\teta)) - \Gamma\varphi\circ\Theta(\teta,\rho)}\\
		&\leq \Lip R\,\abs{(\teta,\rho) - (\teta,\varphi(\teta))} + \Lip\Gamma\varphi\abs{\Theta(\teta,\varphi(\teta)) - \Theta(\teta,\rho)},
	\end{align*}
	from the definition of $\Gamma$. We observe that$$\abs{\Theta(\teta,\varphi(\teta)) - \Theta(\teta,\rho)}\leq \pa{\frac{1-e^{-2\pi\eta}}{\eta} + \eps A_f}\abs{\varphi(\teta) - \rho}\leq (2\pi + \eps A_f)\, \abs {\varphi(\teta) - \rho},$$ hence
	$$ \abs{R(\teta,\rho) - \Gamma\varphi\circ\Theta(\teta,\rho)}\leq \sq{\Lip R + \Lip\Gamma\varphi\,(2\pi + \eps A_f)} \abs{\varphi(\teta) - \rho},$$ and this chain of inequalities holds 
	\begin{align*}
		\Lip R + \Lip\Gamma\varphi\,(2\pi +\eps A_f) &\leq \Lip R + k(2\pi + \eps A_f) \\
		&\leq e^{-2\pi\eta} + \eps A_g + k2\pi + \eps k A_f\\
		& = 1 - 2\pi\eta + O(\eta^2) + k 2\pi + \eps A_g + \eps k A_f < 1.
	\end{align*}
\end{proof}

\textit{Step 3: Existence of the invariant curve.}
	We want to show that $\Gamma$ defines a contraction in the space $\Lip_{k}$: indeed $\Lip_k$ is a closed subspace of the Banach space $C^0(\T,[-1,1])$, hence complete. The standard fixed point theorem then applies once we show that $\Gamma$ is a contraction.\\ Let $z$ be a point of $\T$, for every $\varphi_1,\varphi_2$ in $\Lip_k$ we want to bound \[\abs{\Gamma\varphi_1(z) - \Gamma\varphi_2(z)}.\]
	The trick is to introduce the following point in $\T\times [-1,1]$, $$(\teta, \rho) = \pa{[\Theta\circ(\id,\varphi_1)]^{-1}(z), \varphi_1 ([\Theta\circ (\id,\varphi_1)]^{-1})(z)}$$ and remark the following equality
	\begin{align*}
		\Gamma\varphi_2\circ\Theta(\teta,\rho) &= \Gamma\varphi_2\pa{\Theta\pa{[\Theta\circ(\id,\varphi_1)]^{-1}(z)\,,\,\varphi_1 ([\Theta\circ (\id,\varphi_1)]^{-1})(z)}}\\
		&= R\circ (\id,\varphi_2)\circ [\Theta\circ(\id,\varphi_2)]^{-1}\circ [\Theta\circ(\id,\varphi_1)][\Theta\circ(\id,\varphi_1)]^{-1}(z)\\
		&= R\circ (\id,\varphi_2)\circ [\Theta\circ(\id,\varphi_2)]^{-1}(z) = \Gamma\varphi_2(z).
	\end{align*}
	
	We hence apply lemma \ref{lemma per contrazione} to $\varphi= \varphi_1$ at the point $(\teta,\rho)$ previously introduced. We have \[\abs{\Gamma \varphi_1(z) - \Gamma\varphi_2(z)} \leq C\, \abs{\varphi_1\circ [\Theta\circ (\id,\varphi_1)]^{-1}(z) - \varphi_2 \circ [\Theta\circ (\id,\varphi_1)]^{-1}(z)}.\]
	Taking the supremum for all $z$ and remembering that $C<1$, concludes the proof of the theorem.

\section{Second localization: Translated circle of rotation $2\pi\alpha$}
\label{second localization}
We now consider the part of the $(\eta,\nu)$-plane defined by $\abs{\eta}\ll1$ in which the graph transform does not work a priori and perform convenient coordinate's changes on $Q$ in order to write it in a normal form that will allow to conclude again the existence of an invariant circle via a second, straighforward, application of the graph transform. While in section \ref{Normally hyperbolic regime} we have localized our study at the circle $\rho=0$, we now want to focus on the one possessing a Diophantine rotation $2\pi\alpha$. In what follows, we make strong use of the ``translated curve theorem" of R\"ussmann that now we recall. 

Let $\alpha\in\R$ satisfy the following Diophantine condition for some $\gamma,\mathtt{q}>0$
\begin{equation}\label{diophantine}
	\abs{k\alpha - l}\geq \frac{\gamma}{\abs{k}^\mathtt{q}}\,\quad \forall (k,l)\in\mathbb{N}\setminus{\set{0}}\times\Z\,.
\end{equation}

\begin{theorem}[R\"ussmann]
	\label{Russmann theorem} Let $\alpha$ satisfy \eqref{diophantine}, $A\in\R$ and suppose that $P^0$ is of the form $$P^0 (\teta,r)=(\teta + 2\pi\alpha + t(r) + O(r^2), (1 + A)r + O(r^2))$$ where $t(0)=0$ and $t'(r)>0$.\\  If $Q'$ is close enough to $P^0$ there exists a unique real analytic curve $\gamma: \T \to \R$, close to $r=0$, a real analytic diffeomorphism $h$ of $\T$ isotopic to the identity and fixing the origin, and $\lambda\in\R$, close to $0$, such that \[Q' (\teta, \gamma(\teta)) = (\conj(\teta), \lambda + \gamma(\conj(\teta))).\]
\end{theorem}

\begin{remark}\label{neuroni sedati}
Note that in \cite{Massetti:APDE} this theorem is proved in the analytic setting; real analytic diffeomorphisms in a neighborhood of $\T\times\set{0}$ admit a (unique) complex-extension and we work on the corresponding spaces of diffeomorphisms, endowed  with a Banach norm. The closeness conditions appearing
in the statement have thus to be understood as referred to that norm.\\
Note that, when $A\neq 0$, $\T\times \set{0}$ is a normally hyperbolic invariant circle of rotation $2\pi\alpha$ of $P^0$. In its original version the theorem is stated for $A=0$ in the context of finitely differentiable diffeomorphisms (see \cite{Russmann:1970} or the work of Herman \cite{Herman:1983} in the case of rotation numbers of constant type). For a proof in $C^\infty$ category  see\cite{Yoccoz:Bourbaki, Bost:1985} for example. In \cite{Massetti:APDE}, we treated in the analytic setting this (straightforward) case $A\in\R$, and recover this theorem as the particular case of a general normal form theorem (Theorem 5.4 of the same article) whose proof is based on a suitable inverse function theorem in analytic category \cite[Theorem A1]{Massetti:APDE}, proved by means of a Newton-like scheme which converges uniformly. As observed in the proof, in the case when $P^0$, $Q'$ depend smoothly on some parameters, the parametrized version of the inverse function theorem follows straightforwardly with the smoothness of $\gamma, h, \lambda$ w.r.t. the parameters. See subsection \ref{russ gen} of Appendix \ref{Russmann theorem appendix} where a statement of the general Theorem 5.4 (recalled here as Theorem \ref{Russmann general}) and a brief discussion is given. 
\end{remark}

\smallskip
We shall apply this theorem in our frame. Let $\alpha$ be Diophantine.
First, note that the translation function $\tau=2\pi\eta(\nu-\alpha)$ defines a family of hyperbolas in the $(\eta,\nu)$-plane (as opposed to the introduction, we dropped the sub-index $\alpha$ in the notation of $\tau$) and observe that the diffeomorphism $P$ in terms of $(\tau,\eta)$ becomes  
\begin{equation*}
	P(\teta,\rho)= (\teta + 2\pi\alpha + \frac{\tau}{\eta} + \frac{1-e^{-2\pi\eta}}{\eta}\,\rho ,\, \rho e^{-2\pi\eta})
\end{equation*}
which after the change of variables 
\begin{equation*}
	(\teta,\rho)\mapsto \pa{\teta, \rho - \frac{2\pi\eta (\nu - \alpha)}{e^{-2\pi\eta} - 1}= r }
\end{equation*}
reads 
\begin{equation}
\label{P russmann}
	P(\teta,r)=\pa{\teta + 2\pi\alpha + \frac{1- e^{-2\pi\eta}}{\eta}\,r \,,re^{-2\pi\eta} + \tau }\,.
\end{equation}

\noindent
If $\nu=\alpha$ then $r = 0$ is the invariant circle of rotation $2\pi\alpha$. 

\noindent
The corresponding perturbation reads
\begin{equation}\label{Q russmann}
	Q(\teta,r)= \pa{\teta + 2\pi\alpha + \frac{1- e^{-2\pi\eta}}{\eta}\,r + \eps f(\teta,r),\,r e^{-2\pi\eta} + \tau + \eps g(\teta,r)}.
\end{equation}


If\footnote{ Note that $P$ and $Q$ have a twist also for small values of $\eta$, $t'(r) = \frac{1-e^{-2\pi\eta}}{\eta}\to 2\pi$ when ${\eta\to 0}$ } $Q$ is close enough to $P$,  by R\"ussmann's theorem there exist a real analytic curve $\gamma:\T\to \R$, a diffeomorphism of the torus $h(\teta) = \teta + O(\eps)$ close to the identity and $\lambda\in\R$, close to $\tau$, such that
\begin{itemize}[leftmargin=*]
	\item the image of the curve $\tilde\rho=\gamma(\teta)$ via $Q$, is the translated curve of equation $r= \lambda + \gamma(\teta)$, where $\lambda = \lambda(\alpha,\nu,\eta,\eps) = \tau + \int_\T \gamma(\teta)d\teta = \tau + O(\eps)$
	\item the restriction of $Q$ to $\Gr\gamma$ is conjugated to the rotation $R_{2\pi\alpha}: \teta\mapsto \teta + 2\pi\alpha$. 
\end{itemize}

\begin{remark}[Curves of invariant diophantine circles]
\label{rmk parameters}
As direct consequence of classical implicit function theorem and the smooth dependence  of $\lambda$ on the parameters, one proves that there exists a positive constant $M$ such that whenever $\eta > M\eps$, there exists $\nu$ such that $\lambda(\alpha,\nu,\eps,\eta) = 0$ and the graph of $\gamma$ is then invariant, of rotation $2\pi\alpha$. The consequence is that for any $\alpha$ Diophantine and any admissible $\eps$ such that R\"ussmann's result applies, in the associated plane $(\nu,\eta)$ in the $(\eps,\eta,\nu)$-space, there corresponds a curve $\cC_\alpha$ $(\eps,\eta)\mapsto \nu(\eps,\eta)$ such that $\lambda(\alpha,\nu(\eps,\eta),\eps,\eta) = 0$.  Along these curves $\nu = \alpha + O(\eps)$ and the  circle of the corresponding diffeomorphism is quasi-periodic of rotation $2\pi\alpha$.  The perturbation $Q$ being quite general, there is a priori no reason why the existence of such curves is guaranteed for $\eta\to 0$, as it happens for the particular case of Hamiltonian perturbations \cite{Massetti:ETDS} (for the interested reader, we recall some of these results in Appendix \ref{Russmann theorem appendix}, where picture should be compared with the ones below).
\end{remark}

\begin{figure}[h]
\centering
\vskip-2mm
\includegraphics[width=.5\textwidth]{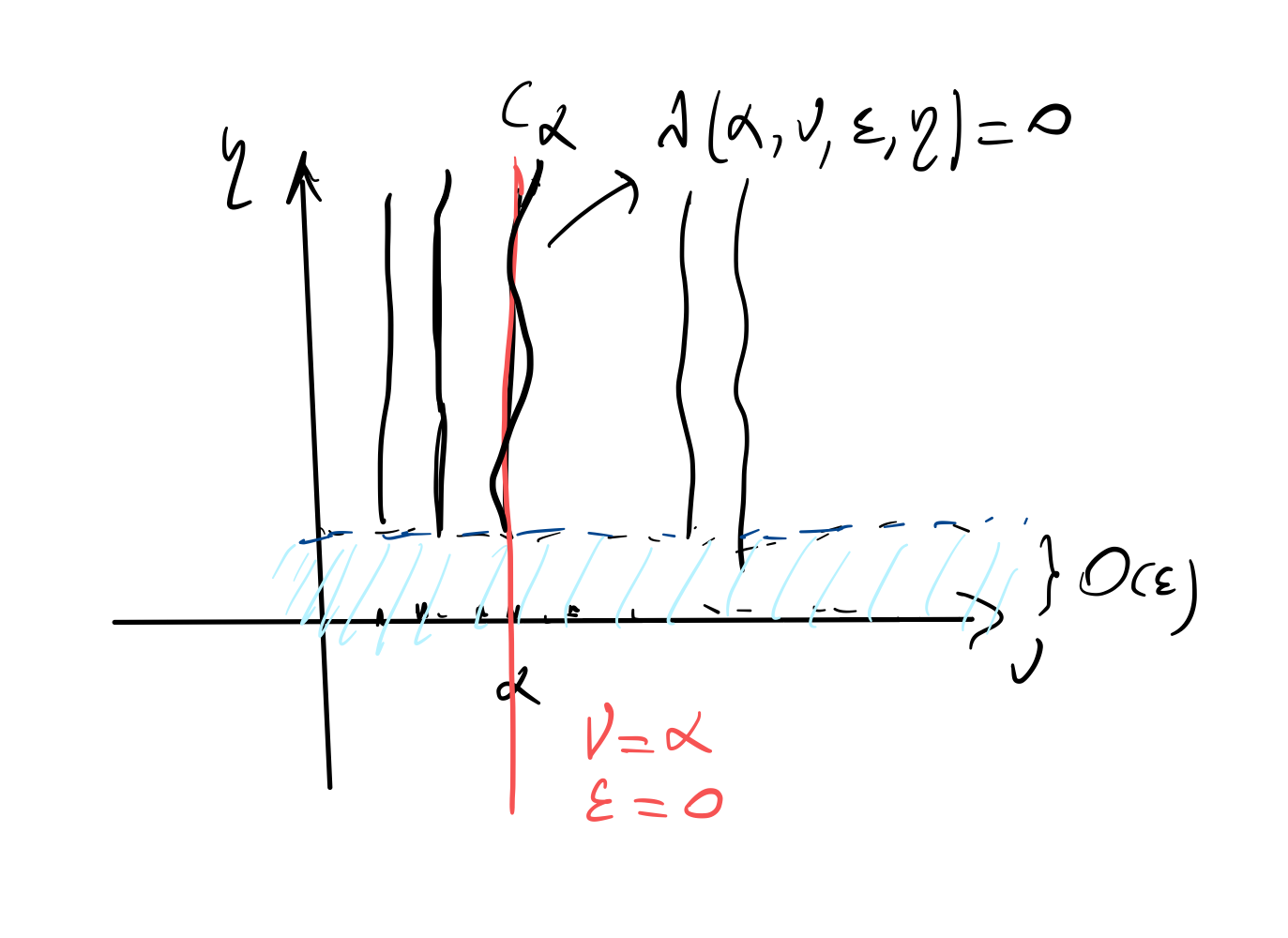}
\vskip-10mm
\caption{Fixing an admissible $\eps$, we see the Cantor set of curves $\cC_\alpha$ in the plane $(\nu,\eta)$. We restricted the picture to the positive plane $\eta>0$, the negative one being similar.}\label{Cantor curves}
\end{figure}

\subsection{Localization at the translated circle}\label{second localization a 0}
We shall now prove that, for parameters in the vicinity of $\cC_\alpha$, there still exists an invariant circle (the rotation of which will no more be controlled), for values of $\eta$ smaller than the ones given in theorem \ref{theorem A} ($1\gg\eta\gtrdot\sqrt{\eps}$).

 \smallskip

Let us first remark that there exists a positive constant $\tilde{C}$ such that outside the annulus $$A^{-}= \set{(\teta,r),\abs{r}\leq \frac{\tau}{1 - e^{-2\pi\eta} } + \tilde{C}\eps/\eta},$$ which contains the translated circle $r= 0$ of $P$,  we have $r' - r< 0$.\\ In order to conclude the existence of an invariant circle inside of it, we first localize at the translated circle of $Q$ close to $r = 0$ (whose existence is guaranteed by Theorem \ref{Russmann theorem}), and write $Q$ in a more refined form that allows to take advantage of the normal attraction also inside of $A^{-}$ so that the graph transform can be applied again. 

R\"ussmann's theorem guarantees the existence of $\eps_0>0$ such that for any $\eps\le\eps_0$,   the local diffeomorphism
\begin{equation*}
	G:(\teta,r)\mapsto \pa{h^{-1}(\teta)=\xi,\, r - \gamma(\teta)= x}, 
	\label{Russmann variables}
\end{equation*}
sends $r = \gamma(\theta)$ to $x = 0$ and is such that $G\circ Q\circ G^{-1}$ has $x=0$ as a translated curve on which the dynamics is the rotation of angle $2\pi\alpha$. 
We have:

\begin{equation*}
	\eqsys{
		Q(\xi,x)=(\xi',x')\\
		\xi'= h^{-1}\pa{h(\xi + 2\pi\alpha) + \frac{1-e^{-2\pi\eta}}{\eta}\,x + \eps\sum_{j=1} \frac{1}{j!}\frac{\partial^j f}{\partial r^k}(\teta,\gamma(\teta))\,x^j},\\
		x'= \lambda + xe^{-2\pi\eta} + \gamma (h(\xi + 2\pi\alpha)) - \gamma(h(\xi + 2\pi\alpha) + \frac{1-e^{-2\pi\eta}}{\eta}x + O(\abs{x})) + \\ + \eps\sum_{i=1} \frac{1}{i!} \frac{\partial^i g}{\partial r^i}(\teta,\gamma(\teta)) x^i,\\
	}
\end{equation*}
hence 
\begin{equation}
	Q(\xi,x) = (\xi + 2\pi\alpha + \sum_i A_i(\xi)x^{i},\, \lambda + \sum_i B_i(\xi)x^i),
	\label{Russmann normal form} 
\end{equation}
where 
\begin{itemize}[label=$\bullet$,leftmargin=*]
	\item $B_1(\teta) = e^{-2\pi\eta} - D\gamma(h(\xi + 2\pi\alpha))\cdot (\frac{1-e^{-2\pi\eta}}{\eta} + \eps\frac{\partial f}{\partial\tilde\rho}(\teta,\gamma(\teta))) + \eps \frac{\partial g}{\partial\tilde\rho}(\teta,\gamma(\teta)) $, hence it is of order $1 + O(\eps)$,
	\item $B_i(\teta)$, for $i>1$,  is the coefficient of the order-$i$ term in $x$ from the development of terms as $$-\frac{1}{i!} D^i\gamma(h(\xi + 2\pi\alpha))\cdot(\frac{1-e^{-2\pi\eta}}{\eta}\,x + \eps\sum_{j=1} \frac{1}{j!}\frac{\partial^j f}{\partial r^j}(\teta,\gamma(\teta))\,x^j )^i + \eps \frac{1}{i!}\frac{\partial^i g}{\partial r^i}(\teta,\gamma(\teta)), $$ and has order $O(\eps)$.
	\item $A_i(\teta)$ is the order-$i$ term coming from 
	\[\frac{1}{i!}D^i h^{-1}(h(\xi + 2\pi\alpha))\cdot (\frac{1-e^{-2\pi\eta}}{\eta}\,x + \eps\sum_{j=1} \frac{1}{j!}\frac{\partial^j f}{\partial r^j}(\teta,\gamma(\teta))\,x^j)^i.\]
\end{itemize}
In particular $A_i(\teta)$ is of order $1 + O(\eps)$ for $i=1$ and $O(\eps)$ otherwise, where we denoted $\teta= h(\xi)$ and omitted indexes indicating the smooth dependence on $\alpha,\eta$ and $\tau$. \\ This change of coordinates actually enables us to write $Q$ as the composition of a map $$\mathcal{I}_{\eta,\tau}=(\xi + 2\pi\alpha + \sum_i A_i(\xi)x^{i},\, \sum_i B_i(\xi)x^i),$$ leaving the circle $x=0$ invariant, with a translation $T_\lambda: (\teta,r)\mapsto (\teta, \lambda + r)$ in the normal direction.  Remark that when $\eps=0$, we have $h=\id$, $\gamma = 0$ and $\lambda = \tau$, thus $Q$ would read as before the perturbation; in addition even if we don't dispose of the explicit form of the translation function $\lambda$, the implicit function theorem tells us that $\lambda = \tau + O(\eps)$. 
\subsection{A normal form for $Q$ and Proof of Theorem \ref{theorem NH2}} If $\lambda\neq 0$, it seems impossible to write $Q$ in a form as gentle as $\mathcal{I}_{\eta,\tau}$. The idea is to use all the strength of the translation $\lambda$: we perform coordinates changes that push the dependence on the angles as far as possible, let say up to a certain order $k$, and eventually remark that \gr{all the dependence on the angles of the remaining terms will cancel out with $\lambda$}. The following result is a key step for the proof of Theorem \ref{theorem NH2}. 

\begin{proposition}
	\label{NF}
	 If $\alpha$ is Diophantine, there exists a real analytic change of variables that transforms $Q$ into the normal form:
	\begin{equation}
		\eqsys{   Q(\Theta,R)= (\Theta',R')\\
			\Theta'= \Theta + 2\pi\alpha + \sum_{i=1}^k \bar\alpha_i R^i + O(\eps \abs{R}^{k+1}) + O(\abs{\lambda}\eps)\\
			R' = \lambda + \bar\beta_1\,R + \sum_{i=2}^k \bar\beta_i\, R^i + O(\eps\abs{R}^{k+1})  + O(\abs{\lambda}\eps),
		}   
	\end{equation}
	$\bar\alpha_i$ and $\bar\beta_i$ being constants, $O(\eps\abs{R}^{k+1})$ being terms of order $\geq k +1$ in $R$, $O(\abs{\lambda}\eps)$ being terms of order $O(\tau\eps) + O(\eps^2)$, possibly depending on $\Theta$ but vanishing when $\lambda = 0$. 
\end{proposition}

\begin{proof}

In order to write $Q$ in a normal form like \eqref{NF}, we do an extensive use of the Diophantine property of $\alpha$ and we repeatedly apply lemma \ref{lemma cohomological circle}.\\
Using the fact that $B_1(\xi)$ is close to $1$, we see that the difference equation 
\begin{equation}
	\log B_1(\xi) + \log X(\xi) - \log X(\xi + 2\pi\alpha) = \frac{1}{2\pi}\int_0^{2\pi}{\log B_1(\xi)}\,d\xi
	\label{difference equation}
\end{equation}
then has a unique analytic solution $X(\xi)$ close to $1$. \\ Hence, the coordinates change  
\begin{equation*}
	(\xi,x)\mapsto (\xi, \frac{x}{X(\xi)}= y)
	\label{log variables}
\end{equation*}
transforms $Q$ into a map of the form 
\begin{equation*}
	\eqsys{
		Q(\xi,y)= (\xi',y')\\
		\xi'= \xi + 2\pi\alpha + \sum_{i=1}^k \alpha_i(\xi)y^i + O(\eps\abs{y}^{k+1})\\
		y' = \lambda + \bar\beta_1\,y + \sum_{i=2}^k \beta_i(\xi)y^i + O(\eps\abs{y}^{k+1})  + O(\eps\abs{\lambda}\abs{y}) + O(\eps\abs{\lambda}),
	}
\end{equation*}
where 
\begin{equation*}
	\bar\beta_1 = \exp \frac{1}{2\pi}\int_0^{2\pi}{\log B_1(\xi)}\,d\xi = 1- 2\pi\eta + 2\pi^2\eta^2 + \eps M_1 +  \eps^2 M_2 + O(\eps\eta) + O(\eps^3,\eta^3),
	\label{beta1} 
\end{equation*}
$M_i$ being constants coming from the average of the order-$\eps^i$ terms in the Taylor's expansion of $\log B_1(\xi).$ \\
Just as for \eqref{difference equation}, there is a unique analytic solution $X^{(2)}(\xi)$, smoothly depending on the parameters - through $\bar\beta_1$ -, of the equation
\begin{equation*}
	\bar\beta_1^2 X^{(2)}(\xi + 2\pi\alpha) - \bar\beta_1 X^{(2)}(\xi) + \beta_2(\xi)= \frac{1}{2\pi}\int_0^{2\pi}\beta_2(\xi)\,d\xi = \bar\beta_2,
\end{equation*}
(where $\bar\beta_1^2$ is the $2$-power of $\bar\beta_1$).
The change of variables 
\begin{equation*}
	(\xi,y)\mapsto (\xi, y + X^{(2)}(\xi)\,y^2)
\end{equation*}
then transforms the non constant coefficient $\beta_2(\xi)$ into its average $\bar\beta_2$. \\
Generalizing, by composing the following changes of variables 
\begin{equation*}
	\eqsys{(\xi,y)\mapsto (\xi, y + X^{(i)}(\xi)\,y^i)\quad i=2,\cdots, k\\
		\bar\beta_1^i\,X^{(i)}(\xi + 2\pi\alpha) - \bar\beta_1\,X^{(i)}(\xi) + \beta_i(\xi) = \bar\beta_i}
\end{equation*}
and
\begin{equation*}
	\eqsys{(\xi,y)\mapsto (\xi + Z^{(i)}(\xi)\,y^i, y)\quad i=1,\cdots, k\\
Z^{(i)}(\xi + 2\pi\alpha) - Z^{(i)}(\xi) + \alpha_i(\xi) = \bar\alpha_i}
\end{equation*}
and denoting with $\Theta, R$ the final variables, $Q$ is eventually in the form entailed in \eqref{NF}
\begin{equation*}
	\eqsys{   Q(\Theta,R)= (\Theta',R')\\
		\Theta'= \Theta + 2\pi\alpha + \sum_{i=1}^k \bar\alpha_i R^i + O(\eps \abs{R}^{k+1}) + O(\abs{\lambda}\eps)\\
		R' = \lambda + \bar\beta_1\,R + \sum_{i=2}^k \bar\beta_i\, R^i + O(\eps\abs{R}^{k+1})  + O(\abs{\lambda}\eps),
	}   
	\label{P normal form order k}   
\end{equation*}
where $\bar\alpha_1$ and $\bar\beta_1$ are of order $1+O(\eps)$ while $\bar\alpha_i,\bar\beta_i$  for $i>1$, of order $O(\eps)$.\\  We thus have been able to confine the angle's dependency entirely in the terms $O(\cdots)$; in particular the terms $O(\abs{\lambda}\eps)$ vanish when no translation occurs. 

\end{proof}

\begin{proof}[Proof of Theorem \ref{theorem NH2}]
Recall that we are now in the regime $|\eta|\lessdot \sqrt{\eps}$. To see that whenever $\eta \gtrdot \eps$ and $\abs{\tau}\leq \eta^2$, the diffeomorphism $Q$ possesses an invariant normally-attractive (resp. repulsive if $\eta<0$) circle, it just remains to harvest the consequences of the normal form \eqref{NF}.\\ 
In new coordinates, the diffeomorphism $Q$ is a perturbation of the normal form
\begin{equation*}
	N_{\eta,\tau}(\Theta,R) = (\Theta + 2\pi\alpha + \sum_{i=1}^k\bar\alpha_i R^i, \lambda (\tau,\eps) + \sum_{i=1}^k \bar\beta_i\,R^i),
\end{equation*}
which possesses an invariant circle $R=R_0$, solution of $R = \lambda + \sum_{k=1}\bar\beta_iR^i$. \\ By the implicit function theorem and the structure of the terms $\bar\beta_1$ and $\bar\beta_2$,  we have 
\begin{equation} R_0 = \frac{-\lambda}{\bar\beta_1 - 1} + O(\frac{\abs{\lambda}^2\abs{\bar\beta_2}}{\abs{\bar\beta_1 -1}^3})= R_{-} + O(\eps \frac{\abs{R_{-}}^2}{\abs{\bar\beta_1 - 1}}),
	\label{invariant circle}
\end{equation}
where $R_{-}$ reads more explicitly as $$R_{-} =\frac{-\tau + O(\eps)}{-2\pi\eta + \eps M_1 + O(\eps\eta) + O(\eta^2,\eps^2)}. $$
In order to apply the graph transform method for the existence of a normally hyperbolic invariant circle close to $R_0$,  we perform a last change of variables:
\begin{equation*}
	(\Theta,R)\mapsto (\Theta, R-R_0 = \tilde{R}).
\end{equation*}
Now centered at $R_0$, the diffeomorphism $Q$ reads
\begin{equation*}
	\eqsys{   Q(\Theta,\tilde R)= (\Theta',\tilde R')\\
		\Theta'= \Theta + 2\pi\alpha + \bar\alpha_1 R_0 + \sum_{i=1}^k \bar\alpha_i \tilde{R}^i + O(\eps\abs{R_0}\abs{\tilde{R}}) + O(\eps\abs{\tilde{R}}^{k+1}) + O(\eps\abs{R_0}^2) + O(\eps\abs{\lambda}) \\
		\tilde R' = (\bar\beta_1 + \sum_{i=2}^k i\,\bar\beta_i\, R_0^{i-1})\tilde{R} + O(\eps\abs{R_0}\abs{\tilde R}^2) + O(\eps\abs{\tilde R}^2) + O(\eps\abs{R_0}^2)  + O(\eps\abs{\lambda}).
	}   
\end{equation*}
Now $\tilde{R} = 0$ is the invariant circle of the normal form.
More explicitly the order $1$ term in $\tilde{R}$ reads
\begin{equation}
	\tilde{R}' = (1 - 2\pi\eta + \eps M_1 + O(\eps\eta) + \sum_{i=2}^k i\,\bar\beta_i\, R_0^{i-1})\,\tilde R + O(\cdots)\,.
	\label{R tilde'}
\end{equation} 
In the region defined\footnote{Note that  each region of this type actually contains the curve $C_\alpha$ along which $\nu = \alpha + O(\eps)$. } by 
\begin{equation}
	\eqsys{\eta \geq \sqrt{2\pi}\abs{\nu - \alpha},\quad \text{hence}\quad\abs{\tau}\leq\eta^2\\
		\eta\gtrdot\eps,}
\end{equation}
the term $\bar\beta_2 R_0$ is of order $O(\eps \eta) + \frac{O(\eps^2)}{O(\eta)}\sim\eps$.\\ 
In the region defined by $\abs{\tau}\leq \eta^2$ and $\eta\gg\eps$, $R_0$ is of order $O(\eps) + O(\eps/\eta)$ and $O(\lambda)= O(\eta^2) + O(\eps)$; more specifically, 
\comment{defining 
	\begin{align*}
		\alpha(\eta;\eps) &= 2\pi\alpha + \bar\alpha_1 R_0\\
		\beta(\eta; \eps) &= 2\pi\eta + O(\eps)\\
		o(\beta(\eta;\eps)) &= O(\eps\abs{R_0}^2) + O(\eps^2) + O(\eps\eta^2),
	\end{align*}
	\begin{equation*}
		\eqsys{   Q(\Theta,\tilde R)= (\Theta',\tilde R')\\
			\Theta'= \Theta + \alpha(\eta;\eps) + (C + O(\eps))\tilde{R} + O(\eps\abs{\tilde{R}}^2)  + o(\beta(\eta;\eps)) \\
			\tilde R' = (1 - \beta(\eta;\eps))\tilde{R} + O(\eps\abs{\tilde{R}}^2) + o(\beta(\eta;\eps)),
		}  
\end{equation*}}

\begin{equation*}
	\eqsys{   Q(\Theta,\tilde R)= (\Theta',\tilde R')\\
		\Theta'= \Theta + 2\pi\alpha + \bar\alpha_1 R_0 + (C + O(\eps))\tilde{R} + O(\eps\abs{\tilde{R}}^2)  + O(\eps\abs{R_0}^2) + O(\eps^2) + O(\eps\eta^2) \\
		\tilde R' = (1 - 2\pi\eta + O(\eps))\tilde{R} + O(\eps\abs{\tilde{R}}^2) + O(\eps\abs{R_0}^2) + O(\eps^2) + O(\eps\eta^2),
	}  
\end{equation*}
having denoted by $C$ the twist $\frac{1-e^{-2\pi\eta}}{\eta}$ and one can see $Q$ as $Q = \tilde{N} + P^\eps,$ where
\begin{equation*}
	\eqsys{   \tilde{N}(\Theta,\tilde R)= (\Theta',\tilde R')\\
		\Theta'= \Theta + 2\pi\alpha + \bar\alpha_1 R_0 + (C + O(\eps))\tilde{R} + O(\eps\abs{\tilde{R}}^2) \\
		\tilde R' = (1 - 2\pi\eta + O(\eps))\tilde{R} + O(\eps\abs{\tilde{R}}^2).
	}  
\end{equation*}
Note that the term $O(\eps\abs{R_0}^2)$ of the perturbation $P^\eps$ is constant.
In the defined region of the parameters, it is now straightforward the application of the graph transform method in the annulus $\abs{\tilde{R}}\leq 1$, containing the circle $R = R_0$. The preponderance of $1-2\pi\eta$ with respect to the other terms, with nontrivial reminders of smaller size, makes the compatibility conditions \eqref{condition GT1} satisfied. The existence of a normally-attractive  circle in a neighborhood of $R_0$ then follows. 

\end{proof}

\appendix

\section{On the elimination of parameters}
\label{Russmann theorem appendix}

In this section we discuss more in detail the content of Remark \ref{rmk parameters} and the elimination of parameters technique applied to the question of persistence of an invariant normally hyperbolic circle of a given Diophantine rotation. Since the perturbing terms considered here are just characterized as analytic functions with uniformly bounded derivatives, there is no reason why such persistence must hold for any choices of $\eta$ uniformly w.r.t. $\eps$, as it may happen in some special cases. We shall give a brief review on this matter, for the sake of clarity and completeness.
 
$\bullet$ In the context of Hamiltonian real analytic perturbations, in the frame of the following $n$-parameter family of dissipative twist vector fields on $\T^n\times \R^n$
\begin{equation}\label{dissip vf}
u_{\nu} = (\dot\teta = \alpha + r, \quad \dot r = -\eta I\cdot r - \eta(\nu -\alpha))\,,\quad \quad (\nu,\eta)\in\R^n\times\R,
\end{equation}
 where  ${I}$ is the identity matrix, the persistence problem has a positive, strong answer. More specifically, the vector field above leaves invariant the torus $T_0=\T^n\times \set{r=0}$, whose tangential dynamics is $\alpha$-quasiperiodic, while its normal dynamics is given by the linear term $-\eta r\partial_r$. If a Hamiltonian perturbation is considered $u_{\nu} + \eps X_H$, the dissipation $-\eta r\partial_r$ is  the unique non-Hamiltonian term and,  provided $\alpha$ is Diophantine, for a unique choice of the external parameter $\nu$, the survival of the reducible torus $T_0$ can be proved by means of a symplectic KAM scheme, thanks to the special structure ``Hamiltonian + homotecy in $r$-direction" that makes these systems conformally symplectic and the scheme smooth (with uniform bounds) in $\eta$, this allowing $\eta\to 0$ smoothly, reaching the symplectic regime. See for instance \cite{Locatelli-Stefanelli:2012, Massetti:ETDS} or \cite{Calleja-Celletti-delaLlave:2013} in the frame of conformally symplectic maps.\\
 $\bullet$ In the particular case of $1+\frac12$ degrees of freedom corresponding to a non-autonomous Hamiltonian perturbation $\eps X_H=(0, -\eps\partial_\teta f(\teta,t))$ in  $\T\times\R$, the system $u_\nu + \eps X_H$ models the well known dissipative spin-orbit problem evoked in the introduction.  An $\alpha$-diophantine torus persists as a global attractor for the dynamics, for a unique choice of $\nu$ and \emph{any} dissipation $\eta$ in an interval containing the origin, see \cite{Celletti-Chierchia:2009}. 
 
 \medskip
 
In \cite{Massetti:ETDS}, the persistence results above has been recovered by means of a technique known as ``elimination of parameters" , in the spirit of Moser, R\"ussmann, Herman, Chenciner and many others \cite{Moser:1967, Fejoz:2004, EFK:2015, Fayad-Krikorian:2009,Herman-Serg, Broer-Huitema-Takens:1990} . This technique is based on a normal form approach, recovering the existence of the invariant quasi-periodic object in two steps which, in terms of the specific example above \eqref{dissip vf} can roughly be summarized as follows: \\
1) Prove that $u_\nu + \eps X_H$ admits a unique, local, normal form as $u_\nu + \eps X_H = g_*u + b\partial_r$,  where $g$ is a diffeomorphism in the neighborhood of the identity in an appropriate subspace of symplectic trasformations, $u$ belongs to the affine subset of vector fields passing through $ (\alpha , -\eta r)$ and directed by Hamiltonian terms $(O(r), O(r^2))$ (i.e. the ones that admit the desired quasi-periodic invariant torus) and $b\in\R^n$ is a counter term in the neighborhood of $\alpha$, used to compensate the possible degeneracy of the system. According to the normal form, the torus $g(T_0)$ is translated by $b$ in the normal direction.\\
2) Since the translation/counter term $b=b(\alpha,\nu,\eta,\eps)$ is proved to depend smoothly on $\nu,\eta$, one may solve implicitly $b(\alpha,\nu,\eta,\eps) = 0$ in terms of $\nu$, thus proving the persistence of $g(T_0)$ for that choice of the parameter, and any $\eta$.
  
This is the content of \cite[Theorem 6.1]{Massetti:ETDS} from which one can deduce the particular $2$-dimensional case and phrase it in the following form, which we rewrite for convenience of the reader.
 
\textbf{Theorem 6.2 and Corollary 6.2 in \cite{Massetti:ETDS}} \textit{Let $\eps_0$ be the maximal value that the perturbation can attain. Every Diophantine $\alpha$ identifies a surface $(\eps,\eta)\mapsto \nu(\eps,\eta)$ in the space $(\eps,\eta,\nu) = [0,\eps_0]\times [-\eta_0,\eta_0] \times \R$, which is analytic in $\eps$, smooth in $\eta$, for which the following holds: for any parameters $(\eps,\eta,\nu(\eps,\eta))$,  the vector field $u_\nu + \eps X_H$ admits an invariant $\alpha$-quasi-periodic torus. This torus is $\eta$-normally attractive (resp. repulsive) if $\eta>0$ (resp. $\eta<0$).\\
	- Fixing $\alpha$ Diophantine and $\eps$ sufficiently small, there exists a unique curve  $C_\alpha$ (analytic in $\eps$, smooth in $\eta$) in the plane $(\eta,\nu)$ of the form $\nu = \alpha + O(\eps^2)$, along which the translation $b = b (\nu,\alpha,\eta,\eps) $ vanishes, so that the perturbed system $u_\nu + (0,\eps\partial_\teta f(\teta,t)) $ possesses an invariant torus carrying quasi-periodic motion of frequency $\alpha$. This torus is normally attractive (resp. repulsive) if $\eta>0$ (resp. $\eta<0$).
}
\begin{figure}[h!]
\begin{picture}(0,0)%
\includegraphics{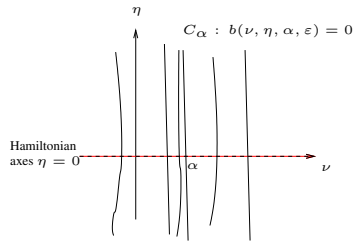}%
\end{picture}%
\setlength{\unitlength}{1657sp}%
\begingroup\makeatletter\ifx\SetFigFont\undefined%
\gdef\SetFigFont#1#2#3#4#5{%
	  \reset@font\fontsize{#1}{#2pt}%
	  \fontfamily{#3}\fontseries{#4}\fontshape{#5}%
	  \selectfont}%
\fi\endgroup%
\begin{picture}(4680,3541)(2146,-5399)
\put(4801,-4321){\makebox(0,0)[lb]{\smash{{\SetFigFont{5}{6.0}{\familydefault}{\mddefault}{\updefault}{\color[rgb]{0,0,0}$\alpha$}%
				}}}}
\put(6811,-4351){\makebox(0,0)[lb]{\smash{{\SetFigFont{5}{6.0}{\familydefault}{\mddefault}{\updefault}{\color[rgb]{0,0,0}$\nu$}%
				}}}}
\put(3976,-1981){\makebox(0,0)[lb]{\smash{{\SetFigFont{5}{6.0}{\familydefault}{\mddefault}{\updefault}{\color[rgb]{0,0,0}$\eta$}%
				}}}}
\put(2161,-4036){\makebox(0,0)[lb]{\smash{{\SetFigFont{5}{6.0}{\familydefault}{\mddefault}{\updefault}{\color[rgb]{0,0,0}Hamiltonian}%
				}}}}
\put(2161,-4261){\makebox(0,0)[lb]{\smash{{\SetFigFont{5}{6.0}{\familydefault}{\mddefault}{\updefault}{\color[rgb]{0,0,0}axes $\eta=0$}%
				}}}}
\put(4741,-2296){\makebox(0,0)[lb]{\smash{{\SetFigFont{5}{6.0}{\familydefault}{\mddefault}{\updefault}{\color[rgb]{0,0,0}$C_\alpha:\,b(\nu,\eta,\alpha,\eps)=0$}%
				}}}}
\label{figure curves}

\end{picture}%
\caption{Unlike the general situation, we can draw these curves up to $\eta=0$}
\end{figure}

\medskip
 In is important to stress that the curves $C_\alpha$ pass continuously through $\eta = 0$ (the unique value of transition between the attractive and repulsive behavior) because of the Hamiltonian nature of the perturbations. Thanks to its special structure and intrinsic symmetries, this model, of great astronomical interest, has been widely investigated in the last decade; its conformally symplectic structure leads to atypical dynamical behaviors that in generic hyperbolic systems cannot happen. Accurate numerical computations have been tailored on that special model for studying the Poincar\'e spin-orbit map, in order to select the parameters in correspondence of which the existence/breakdown of the quasi-periodic attractor may occur,  or for computing possible rotation numbers of its orbits and so on (see for instance \cite{CallejaArxiv} and references therein). 
 
%

On the other hand, under generic perturbation, there is no hope in general for a Theorem as above, since typically the perturbation obstructs the existence of such tori for small value of hyperbolicity. 
 Moreover, in high dimension, whether it is possible to get a similar reducibility result (i.e. up to very small values of dissipation) for a system as \eqref{dissip vf} in the case when the homotetic dissipative matrix is replaced by a general diagonalizable one or under generic perturbations (thus breaking the special spin-orbit like conformal symplecticity symmetries) is open so far. A first step forward is provided in \cite{Massetti:APDE}, in the context of general non-conservative diffeomorphisms of the $2n$-dimensional cylinder, close to having a reducible Diophantine torus and in the recent study \cite{Fayad-Massetti}, in the context of flows that leaves invariant a Diophantine torus of general normal hyperbolicity (i.e. the linear dynamics is given by a diagonalizable matrix $A$ of real eigenvalues).

\subsection{A normal form theorem and the case of general perturbations}\label{russ gen}\label{appendix curves}
For convenience of the reader, we recall here the general setting of \cite{Massetti:APDE} where  R\"ussmann's translated curve theorem is recovered in analytic category, as a special case of the general normal form \cite[Theorem 5.4]{Massetti:APDE}, which we recall below.  

Let $V$ be the space of germs along $\T^n\times\set{0}$ in $\T^n\times\R^m = \set{(\teta,r)}$ of real analytic diffeomorphisms. Fix $\alpha\in\R^n$ and $A\in\GL_m(\R)$, assuming that $A$ is diagonalizable with (possibly complex) eigenvalues $a_1,\ldots,a_m\in\C$. \\Let $U(\alpha,A)$ be the affine subspace of $V$ of diffeomorphisms of the form 
\begin{equation}
P(\teta,r) = (\teta + \alpha + O(r), A\cdot r + O(r^2)),
\label{P0}
\end{equation}
 where $O(r^k)$ are terms of order $\geq k$ in $r$ which may depend on $\teta$. For these diffeomorphisms $\text{T}^n_0 = \T^n\times\set{0}$ is an invariant, reducible, $\alpha$-quasi-periodic torus whose normal dynamics at the first order is characterized by $a_1,\ldots,a_m.$ 

 Let $\cG$ be the space of germs of real analytic diffeomorphisms of $\T^n\times\R^m$ of the form 
\begin{equation}
G(\teta,r) = (h(\teta), R_0(\teta) + R_1(\teta)\cdot r) ,
\label{G}
\end{equation}
where $h$ is a diffeomorphism of the torus fixing the origin and $R_0, R_1$ are functions defined on the torus $\T^n$ with values in $\R^m$ and $\GL_m(\R)$ respectively and such that $\Pi_{\operatorname{Ker} (A - I)} R_0(0)=0$ and $\Pi_{\operatorname{Ker} [A,\cdot]} (R_1(0) - I) =0$, having denoted $I$ the identity matrix in $\Mat_m(\R)$ and $\Pi$ the projection on the subspace in subscript.\\

\begin{theorem}[Theorem 5.4 \cite{Massetti:APDE}]
 \label{Russmann general}
 Let $\alpha$ satisfy $$\abs{k\cdot\alpha - 2\pi l}\geq \frac{\gamma}{\abs{k}^\mathtt{q}},\qquad \forall k\in\mathbb{Z}^n\setminus\set{0},\forall l\in\Z.$$
  On a neighborhood of $\T^n\times\set{0}\subset \T^n\times\R^n$, let $P^0\in U(\alpha, A^0)$ be a diffeomorphism of the form \[P^0(\teta,r) = (\teta + \alpha + p_1(\teta)\cdot r + O(r^2), A^0 \cdot r + O(r^2)) ,\] where $ A^0$ is invertible and has simple, real eigenvalues and such that \[\operatorname{det}\pa{\int_{\T^n} p_1(\teta)\,d\teta}\neq 0.\]
   If $Q'$ is close enough\footnote{Recall Remark \ref{neuroni sedati}.} to $P^0$ there exists a unique $A'$, close to $A^0$, and a unique $(G,P,\lambda)\in\cG\times U(\alpha, A')\times \R^n$, close to $(\id, P^0, 0)$, such that $Q' = T_\lambda\circ G\circ P \circ G^{-1},$ where $T_\lambda(\teta, r) = (\teta, \lambda + r)$.
 \end{theorem}
 Note that,  to be consistent with the notations in present paper, in this statement we denoted with $\lambda$ and $h$ the translation and diffeomorphism that in \cite{Massetti:APDE} is denoted $b$ and $\varphi$ respectively.\\
Phrasing the thesis, the graph of $\gamma = R_0\circ h^{-1}$ is a translated torus on which the dynamics is conjugated to $R_{\alpha}$ by $h$. By stability of the simple real spectrum, the normal dynamics is of the same nature as before, characterized by $A'$ close to $A^0$. \textit{Au passage}, we remark that the spectrum of $A^0$ is used as a set of free (real) parameters\footnote{This spectrum being real, the usual Diophantine conditions on $\alpha$ only are sufficient, with no need of Melnikov's ones.}, the variation of which allows to cancel an additional matricial obstruction counter-term that otherwise would be present in the translation function (see the proof of \cite[Theorem 5.1]{Massetti:APDE} or the more delicate ``th\'eor\`eme de conjugaison hypoth\'etique" of \cite{Fejoz:2004}).

R\"ussmann's theorem \ref{Russmann theorem} recalled in section \ref{second localization} is then a particular case of dimension two where, in order to stress the hyperbolicity character we wrote $1 + A^0$, $A^0\neq 0$, and where we directly gave the statement in terms of $h$ (where $h - \id = \varphi$ of Theorem \ref{Russmann general}, defined in \eqref{G}, $\id$ being the identity map). The solutions  $\lambda,\gamma$ and $h$ of the nonlinear equation $Q' = T_\lambda\circ G\circ P \circ G^{-1}$ are constructed as the inverse of an appropriate nonlinear normal form operator, defined from a convenient neighborhood of $(\id,P^0,0)$  to a neighborhood of $P^0$ by means of a Newton-like scheme which converges uniformly and is smooth w.r.t. the parameters on which the given diffeomorphism $Q'$ smoothly depend (see \cite[Section 2]{Massetti:APDE} and \cite[Appendix A]{Massetti:ETDS}). 
Hence, in the case of our interest, R\"ussmann's theorem gives \[Q (\teta, \gamma(\teta)) = (\conj(\teta), \lambda + \gamma(\conj(\teta)))\] where $\lambda = \tau + O(\eps) = 2\pi\eta(\nu - \alpha) + O(\eps)$, $\varphi = \id + O(\eps)$, $\gamma = O(\eps)$ are analytic in $\eps$, and smooth in $\eta,\nu$.  
In order to prove that $\lambda = 0$ implicitly defines $\nu$, it suffices then to show that $\nu\mapsto \lambda(\nu,\eta,\eps)$ is a local diffeomorphism; since this is an open property and $Q$ is close to $P$, it suffices to show it for $P$, which is immediate. 
In fact, it suffices to observe that for the map $\R^3\ni(\eps,\nu,\eta)\mapsto \lambda(\eps,\nu,\eta)$ at $p_0 = (0,\alpha,\eta)$ gives
	$\lambda(p_0)= 0$ and that $\frac{\partial \lambda}{\partial\nu}_{|_{\eps=0}} = 2\pi\eta >0,$ which will remain so after perturbation for $\eta$ not too close to $0$.\\
	In fact, let now $\eps_0$ be the maximal admissible perturbation for R\"ussmann's theorem to apply, and consider the closed ball $B_{\eps_0}(p_0)$ of radius $\eps_0$ centered at $p_0\in\R^3$. Because of the regularity of $\lambda$ with respect to $\eps,\nu$ and $\eta$, 
	there exists a positive constant $M$ independent of $\eps,\eta,\nu$ such that $\norm{\lambda}_{C^2}< M$. Considering the ball of radius $\eps<\eps_0$, the mean value theorem applied to $\frac{\partial \lambda}{\partial \nu}$ and the triangular inequality yields $\forall p_2,p_1\in B_{\eps/2}(p^0)$ that \[\abs{\frac{\partial \lambda}{\partial \nu}(p_2)}\geq \abs{\frac{\partial \lambda}{\partial\nu}(p_1)} - M\eps.\] Fixing $p_1 = p_0$, a sufficient condition for having $\abs{\frac{\partial \lambda}{\partial \nu}(p_2)}>\pi\eta$ is that $ \pi\eta/ 4 M > \eps.$ \\ Hence, for every fixed value of $\eps$, we can guarantee that the derivative of $\lambda$ with respect to $\nu$ is different from $0$, for those $\eta'$s such that $\eta \geq \eps 4M /\pi$, hence by the implicit function theorem, there exists $\nu$ such that $\lambda(\nu,\eps,\eta) = 0$.

\section{Classical tools}

\subsection{Graph transform lemmata}

\begin{proof}[Proof of Lemma \ref{inversion lemma}]
\begin{enumerate}
Since $f$ and $g$ are real analytic on $\T\times [-1,1]$, they are Lipschitz. 
\item If $u$ is a contraction, $\id + u$ is invertible  with $\Lip (\id + u)^{-1} \leq \frac{1}{1 - \Lip u}$. \\ Letting  $u = \id + \Theta(\id,\varpi)$ and $f$ being analytic, we have 
	\begin{align*}
		\abs{u(\teta_1)-u(\teta_2)} &\leq \Lip\varphi\frac{(1-e^{-2\pi\eta})}{\eta}\abs{\teta_1 - \teta_2} + \eps A_f\abs{(\teta_1,\varphi(\teta_1)) - (\teta_2,\varphi(\teta_2))}\\
		&\leq \pa{ \frac{(1-e^{-2\pi\eta})}{\eta} k + \eps A_f(1 + k)}\abs{\teta_1 - \teta_2},
	\end{align*}
	where $A_f = \max \pa{\sup_{\T\times [-1,1]}\abs{D_\teta f},\sup_{\T\times [-1,1]} \abs{D_r f}}$. \\ Since $\eps,k\ll1$, $\Lip u< 1$.

	\item It is immediate from the expression of $Q$ that for $z_1,z_2$ in $\T\times [-1,1]$
	
	\begin{equation*}
		\abs{R(z_1) - R(z_2)}\leq (e^{-2\pi\eta} + \eps A_g)\abs{z_1 - z_2}
	\end{equation*}
	where $A_g = \max \pa{\sup_{\T\times [-1,1]}\abs{D_\teta g},\sup_{\T\times [-1,1]} \abs{D_r g}}$, and
	\[\abs{\Theta(z_1) - \Theta(z_2)}\leq (1 + \frac{(1-e^{-2\pi\eta})}{\eta} + \eps A_f)\abs{z_1 - z_2}.\]
	\end{enumerate}
\end{proof}

\subsection{Difference equation on the torus}

Consider the complex extension $\T_{\C}= \C/{2\pi\Z}$ of the torus $\T=\R/{2\pi\Z}$ and the corresponding $s$-neighborhood defined using $\ell^\infty$-balls (in the real normal bundle of the torus): \[\T_s=\set{\teta\in\T_{\C}:\, \abs{\Im\teta}\leq s}\]

\comment{Let now $f: \T_s\to \C$ be holomorphic, and consider its Fourier expansion $f(\teta)=\sum_{k\in\Z}\,f_k(r)\,e^{i\,k\cdot\teta}$. In this context we introduce the so called "weighted norm": \[\abs{f}_s := \sum_{k\in\Z}\abs{f_k}\, e^{\abs{k}s}.\] It is a trivial fact that the classical sup-norm is bounded from above by the weighted norm: \[\sup_{\teta\in{\T_s}}\abs{f(\teta)}\leq\abs{f}_s\] and that $\abs{f}_s<+\infty$ whenever $f$ is analytic on its domain, which necessarily contains some $\normal{T}_{s'}$ with $s'>s$. In addition, the following useful inequalities hold if $f,g$ are analytic on $\T_{s'}$ \[\abs{f}_s\leq\abs{f}_{s'}\,\text{ for }\, 0<s<s',\] and \[\abs{fg}_{s'}\leq \abs{f}_{s'}\abs{g}_{s'}.\] For more details about the weighted norm, see for example \cite{Meyer:1975} and \cite{Chierchia:2003}.\\} 
Given a real analytic function, we shall consider its unique complex extension, holomorphic on some torus of a certain complex width $s$. Let $\mathcal{A}(\T_s)$ be the space of holomorphic functions $f : \T_s\to\C$ with finite Banach norm \[\abs{f}_s = \sup_{\T_s} \abs{f(\teta)}.\] 
\begin{lemma} 
	\label{lemma cohomological circle} Let $\alpha$ be Diophantine in the sense of \eqref{diophantine}, $g\in\mathcal{A}(\T_{s+\sigma})$ and let some constants $a,b\in\R\setminus\set{0}$ be given. There exist a unique $f\in\mathcal{A}(\T_{s})$ of zero average and a unique $\mu\in\R$ such that the following is satisfied
	\begin{equation}
		\mu + a f(\teta + 2\pi\alpha) - b f(\teta) = g(\teta),\quad \mu = \frac{1}{2\pi}\int_{\T}g(\teta)d\teta.
	\end{equation}
	In particular $f$ satisfies \[\abs{f}_s\leq \frac{C}{\gamma\sigma^{\tau + 1}}\abs{g}_{s+\sigma},\] $C$ being a constant depending on $\tau$.
\end{lemma} 
The proof is classical and we omit it. We address the reader interested to optimal estimates to \cite{Russmann:1976}.

\subsection*{Acknowledgments} The author is grateful to A. Chenciner for many reasons, of many shades.  This work owes a lot to his support and advise in the years of her PhD, full of deep beautiful mathematics and delicate humanity.  Thanks also to J. F\'ejoz, L. Biasco, L. Chierchia,  A. Pousse, and L. Niederman for helpful remarks and fundamental conversations.

\subsection*{Funding} The author has been  supported by the  research project PRIN 2020XBFL ``Hamiltonian and dispersive PDEs" of the Italian Ministry of Education and Research (MIUR) and by the INdAM-GNAMPA research project ``Chaotic and unstable behaviors of infinite-dimensional dynamical systems".

\subsection*{Declarations} Data sharing not applicable to this article as no datasets were generated or analyzed during the current study.

\noindent
Conflicts of interest: The authors have no conflicts of interest to declare.

\end{document}